\documentclass{amsart} 

\usepackage{fullpage}

\usepackage[usenames,dvipsnames,svgnames,table]{xcolor}

\usepackage{amsmath, amssymb, mathrsfs}

\usepackage{tikz}
 \usetikzlibrary{cd}
 \tikzset{
  symbol/.style={
    draw=none,
    every to/.append style={
      edge node={node [sloped, allow upside down, auto=false]{$#1$}}}
      }
      }

\usepackage{amsthm}
	\theoremstyle{definition} 
	\newtheorem{defn}{Definition}[section]
	
	\theoremstyle{plain} 
	\newtheorem{thm}[defn]{Theorem}
	\newtheorem*{thm*}{Theorem} 
	\newtheorem{lem}[defn]{Lemma}
	\newtheorem{prop}[defn]{Proposition}
	\newtheorem{cor}[defn]{Corollary}
	
	\theoremstyle{remark} 
	\newtheorem{rmk}[defn]{Remark}
        
\usepackage{enumitem}

\renewcommand{\AA}{\mathbb A}

\newcommand{\CC}{\mathbb C}

\newcommand{\HH}{{\mathbb H}}

\newcommand{\PP}{{\mathbb P}}
\newcommand{\QQ}{{\mathbb Q}}
\newcommand{\RR}{{\mathbb R}}
\newcommand{\bS}{{\mathbb S}}

\newcommand{\ZZ}{{\mathbb Z}}

\newcommand{\cA}{{\mathcal A}}

\newcommand{\cC}{{\mathcal C}}
\newcommand{\cD}{{\mathcal D}}
\newcommand{\cE}{{\mathcal E}}
\newcommand{\cF}{{\mathcal F}}

\newcommand{\cH}{{\mathcal H}}

\newcommand{\cJ}{{\mathcal J}}

\newcommand{\cM}{{\mathcal M}}

\newcommand{\cO}{{\mathcal O}}

\newcommand{\cT}{{\mathcal T}}
\newcommand{\cU}{{\mathcal U}}
\newcommand{\cV}{{\mathcal V}}
\newcommand{\cW}{{\mathcal W}}
\newcommand{\cX}{{\mathcal X}}
\newcommand{\cY}{{\mathcal Y}}
\newcommand{\cZ}{{\mathcal Z}}

\newcommand{\frakN}{\mathfrak N}

\newcommand{\frakX}{\mathfrak X}

\newcommand{\frakc}{\mathfrak c}

\newcommand{\frakg}{\mathfrak g}
\newcommand{\frakh}{\mathfrak h}

\newcommand{\frakl}{\mathfrak l}
\newcommand{\frakm}{\mathfrak m}

\newcommand{\frakp}{\mathfrak p}
\newcommand{\frakq}{\mathfrak q}

\newcommand{\sF}{\mathscr F}

\newcommand{\sP}{\mathscr P}

\newcommand{\Qbar}{\overline{\QQ}}

\newcommand{\Qpbar}{\Qbar_p}

\newcommand{\cris}{\mathrm{cris}}

\newcommand{\Sen}{\mathrm{Sen}}

\usepackage{colonequals}
\newcommand{\defeq}{\colonequals} 

\newcommand{\isom}{\cong} 
\newcommand{\congr}{\equiv} 

\newcommand{\inj}{\hookrightarrow}
\newcommand{\surj}{\twoheadrightarrow}

\newcommand{\directsum}{\oplus} 
\newcommand{\tensor}{\otimes} 
\newcommand{\Directsum}{\bigoplus} 
\newcommand{\Tensor}{\bigotimes} 

\DeclareMathOperator\Hom{Hom} 
\DeclareMathOperator\End{End} 
\DeclareMathOperator\GL{GL} 
\DeclareMathOperator\SL{SL} 
\DeclareMathOperator\Sp{Sp} 
\DeclareMathOperator\Or{O} 
\DeclareMathOperator\SO{SO} 

\DeclareMathOperator\Spec{Spec} 
\DeclareMathOperator\Spf{Spf} 

\newcommand{\Gm}{\mathbb{G}_m} 

\DeclareMathOperator\Sh{Sh} 

\newcommand{\rig}{\mathrm{rig}} 
\newcommand{\ord}{\mathrm{ord}} 
\newcommand{\tor}{\mathrm{tor}} 
\newcommand{\Iw}{\mathrm{Iw}} 

\DeclareMathOperator\Res{Res} 
\DeclareMathOperator\Ext{Ext} 

\newcommand{\id}{\mathrm{id}} 

\DeclareMathOperator\val{val} 
\DeclareMathOperator\Gal{Gal} 
\DeclareMathOperator\Tr{Tr} 
\DeclareMathOperator\Nm{Nm} 
\DeclareMathOperator\Frob{Frob} 
\DeclareMathOperator\im{im} 
\DeclareMathOperator\Sym{Sym} 
\let\det\relax
\DeclareMathOperator{\det}{det} 

\newcommand{\kunder}{{\underline{k}}}
\newcommand{\omegaunder}{{\underline{\omega}}}

\newcommand{\alphaunder}{{\underline{\alpha}}}
\newcommand{\sfD}{\mathsf{D}}
\newcommand{\Xbar}{\overline{X}}
\newcommand{\DR}{\mathrm{DR}}
\newcommand{\wt}{\mathrm{wt}}
\newcommand{\cyc}{\mathrm{cyc}}
\newcommand{\der}{\mathrm{der}}
\newcommand{\sa}{\mathrm{sa}} 
\DeclareMathOperator\Fr{Fr} 
\DeclareMathOperator\Fil{Fil}
\DeclareMathOperator\diag{diag}
\DeclareMathOperator\loc{loc}

\usepackage[OT2,T1]{fontenc}
\DeclareSymbolFont{cyrletters}{OT2}{wncyr}{m}{n}
\DeclareMathSymbol{\Sha}{\mathalpha}{cyrletters}{"58} 

\DeclareRobustCommand{\SkipTocEntry}[5]{} 
\usepackage[bookmarks=false, 
        colorlinks, 
        linkcolor=red, 
        anchorcolor=red, 
        citecolor=red, 
        urlcolor=red] 
        {hyperref}  

\newcounter{counter}
\begin{document}
\title{Ramification of Hilbert eigenvarieties at classical points}
\author{Chi-Yun Hsu}
\address{Department of Mathematics, Harvard University, 1 Oxford Street, Cambridge, MA 02138, USA}
\email{chiyun@math.harvard.edu}
\date{\today}


\begin{abstract}
Andreatta--Iovita--Pilloni constructed eigenvarieties for cuspidal Hilbert modular forms. 
The eigenvariety has a natural map to the weight space, called the weight map.
At a classical point, we compute a lower bound of the dimension of the tangent space of the fiber of the weight map using Galois deformation theory.
Along with the classicality theorem due to Tian--Xiao, this enables us to characterize the classical points of the eigenvariety which are ramified over the weight space, in terms of the local splitting behavior of the associated Galois representation.
\end{abstract}

\maketitle

\tableofcontents


\section{Introduction}
Let $F$ be a totally real field of degree $d$ over $\QQ$.
Andreatta, Iovita, and Pilloni constructed the \emph{cuspidal Hilbert eigenvariety} $\cE$ parametrizing $p$-adic overconvergent cuspidal Hilbert Hecke eigenforms of finite slope over $F$ (\cite{AIP16}).
Let $\cW \defeq \Spf(\ZZ_p[\![(\cO_F \tensor_\ZZ \ZZ_p)^\times \times \ZZ_p^\times]\!])^\rig$ be the \emph{weight space}, which parametrizes continuous characters on $(\cO_F\tensor_\ZZ \ZZ_p)^\times$.
Both the cuspidal Hilbert eigenvariety $\cE$ and the weight space $\cW$ are $(d+1)$-dimensional rigid analytic spaces over $\QQ_p$.
Moreover there is a natural map from the cuspidal Hilbert eigenvariety to the weight space, the \emph{weight map}
\[
 \wt \colon \cE \rightarrow \cW,
\]
sending an overconvergent Hecke eigenform to its weight character. 
For example, a classical Hilbert modular form has weight $((k_\tau)_\tau,w)\in\ZZ^{d+1}$, where $\tau$ runs through archimedean places of $F$, and its weight character is $(z,z')\mapsto \left(\prod_\tau \tau(z)^{k_\tau}\right) \cdot z'^w$.

Our purpose is to study the ramification locus of the cuspidal Hilbert eigenvariety over the weight space from the perspective of Galois representations.
Let $x\in \cE$ be a point on the cuspidal Hilbert eigenvariety.
Then $x$ corresponds to $f$, a normalized overconvergent cuspidal Hilbert Hecke eigenform of finite slope.
There is a $p$-adic Galois representation $\rho_x \colon \Gal_F \rightarrow \GL_2(\CC_p)$ associated to $x$ matching the Frobenius eigenvalues of $\rho_x$ with the Hecke eigenvalues of $f$.

Now assume that $p$ splits completely in $F$.
Also assume that $f$ is classical (Definition~\ref{defn:classical}). 
Then the weight of $f$ is $((k_\tau)_\tau, w)\in\ZZ^{d+1}$, where $k_\tau,w\in\ZZ$, $w\geq k_\tau \geq2$ and $k_\tau \congr w \pmod 2$ for all archimedean places $\tau$ of $F$.
One knows that for all primes $\frakp\mid p$ of $F$, the $U_\frakp$-eigenvalue $\lambda_\frakp$ of $f$ satisfies 
\[
 \frac{w-k_{\tau_\frakp}}2 \leq \val_p(\lambda_\frakp) \leq \frac{w+k_{\tau_\frakp}-2}2.
\]
Here $\tau_\frakp$ is the archimedean place of $F$ identified with the $p$-adic place $\frakp$ via a fixed isomorphism $\CC\isom \Qpbar$, and $\val_p$ is the $p$-adic valuation normalized such that $\val_p(p)=1$.
The rational number $\val_p(\lambda_\frakp)$ is called the \emph{$\frakp$-slope} of $f$, and the maximal possible $\frakp$-slope $\frac{w+k_{\tau_\frakp}-2}2$ is called the \emph{critical} $\frakp$-slope.
Our main theorem says that a classical point $x\in \cE$ is ramified with respect to the weight map if and only if there exists $\frakp\mid p$ such that the local Galois representation $\left.\rho_x\right|_{\Gal_{F_\frakp}}$ splits and $x$ has critical $\frakp$-slope.

In fact, we provide two more equivalent statements.
One of the statements involve the theta operator $\Theta$, which is an endomorphism on the space of overconvergent Hilbert modular forms.
See Section~\ref{subsec:theta} for the precise definition of the theta operator.
Here is our main theorem.
\begin{thm} \label{thm:main}
 Assume $p$ splits completely in $F$.
 Let $f$ be a classical cuspidal Hilbert Hecke eigenform of weight $(\kunder,w)$ of finite slope.
 For each prime $\frakp$ of $F$ above $p$, let $\lambda_\frakp$ denote the $U_\frakp$-eigenvalue of $f$, and assume that $\val_p(\lambda_\frakp)\neq \frac{w-1}2$.
 Let $x \in \cE$ be the point on the cuspidal Hilbert eigenvariety $\cE$ corresponding to $f$. 
 Then the following are equivalent.
 \begin{enumerate}
  \item \label{cond:ram} The point $x$ is ramified with respect to the weight map $\wt\colon \cE \rightarrow \cW$.
  \item \label{cond:geneigen} There exists an overconvergent cuspidal Hilbert generalized Hecke eigenform $f'$ with the same Hecke eigenvalues and weight as $f$, but which is not a scalar multiple of $f$.
  \item \label{cond:imtheta} $f$ is in the image of $\Theta$.
  \item \label{cond:split} There exists a prime $\frakp$ of $F$ above $p$ such that the local Galois representation $\left.\rho_x\right|_{\Gal_{F_\frakp}}$ splits and $x$ has critical $\frakp$-slope.
 \end{enumerate}
\end{thm}

We make some remarks about the theorem.
When $F=\QQ$, a conjecture of Greenberg says that a classical cuspidal Hecke eigenform has complex multiplication (CM) if and only if its $p$-adic Galois representation splits locally at $p$. 
Hence assuming the conjecture of Greenberg, Theorem~\ref{thm:main} says that a classical point of critical slope on the eigencurve is ramified if and only it is CM.
This was conjectured by Coleman (\cite[Remark 2 in Sec.\ 7]{Col96}).
In fact, the conjecture of Greenberg was proved by Emerton in weight $2$ (\cite{Emer}).
For general $F$, we can formulate a generalization of the conjecture of Greenberg: a classical cuspidal Hilbert Hecke eigenform is CM if and only if its $p$-adic Galois representation splits locally at all primes of $F$ above $p$.
In the case of parallel weight $2$, the argument of Emerton implies that the Galois representation splitting at one $p$-adic prime of $F$ is CM, and hence also splits at all other $p$-adic primes.
We conjecture that this holds for general weights. 
Geometrically, this means that a classical point $x\in \cE$ of critical slope is ramified if and only if it is CM.

In Theorem~{\ref{thm:main}}, the equivalence between (\ref{cond:ram}) and (\ref{cond:split}) was known for $F=\QQ$.
It was first proven by Breuil--Emerton (\cite[Th\'eor\`eme 1.1.3]{BE10}), an ingredient in their work to prove $p$-adic local-global compatibility for $\GL_2$ over $\QQ$ in the locally reducible case.
Later Bergdall gave a different and simpler proof (\cite{Berg14}).
Our proof for a general totally real field $F$ is a generalization of Bergdall's idea.

The structure of the proof of Theorem~\ref{thm:main} is as follows:
We prove the equivalence of (\ref{cond:ram}) and (\ref{cond:geneigen}) in Lemma~\ref{lem:1and2}.
This is basically unwinding the definitions.
The equivalence of (\ref{cond:geneigen}) and (\ref{cond:imtheta}) is Corollary~\ref{cor:2and3}, strongly relying on a classicality theorem at critical slope deduced from Tian--Xiao's work (\cite{TX16}).
The implication from (\ref{cond:imtheta}) to (\ref{cond:split}) is Proposition~\ref{prop:3to4}.
The key is an argument of companion forms, with a slight complication in the case of totally real fields not equal to $\QQ$ because of the presence of multiple theta operators.
Finally the implication from (\ref{cond:split}) to (\ref{cond:ram}) is a corollary of the following theorem, a computation of the dimension of the tangent space of the fiber of the weight map.

\begin{thm}[Theorem~\ref{thm:4to1}] \label{thm:main2}
 Assume that $p$ splits completely in $F$.
 Let $f$ be a classical cuspidal Hilbert Hecke eigenform of finite slope.
 For each prime $\frakp$ of $F$ above $p$, let $\lambda_\frakp$ denote the $U_\frakp$-eigenvalue of $f$.
 Let $x$ be the point on the cuspidal Hilbert eigenvariety $\cE$ corresponding to $f$. 
 Then the tangent space $T_x\cE_{\wt(x)}$ of the fiber of $\wt$ at $x$ satisfies
\[
 \dim_{\bar{k}(x)} T_x\cE_{\wt(x)} \geq\#\{\frakp \mid \text{the local representation $\left.\rho_x\right|_{\Gal_{F_\frakp}}$ splits and $x$ has critical $\frakp$-slope} \},
\]
where $\bar{k}(x)$ is the residue field of the point $x$.
\end{thm}

The essence of the proof of the theorem is a computation in Galois deformation theory.
To set up a Galois deformation problem characterizing overconvergent Hilbert modular forms, we use the analytic continuation of crystalline periods for overconvergent modular forms.
This is due to \cite{KPX14} and \cite{Liu15} independently, building on the work of Kisin for $F=\QQ$ (\cite{Kis03}).
With a computation in Galois cohomology, the condition ``$\left.\rho_x\right|_{\Gal_{F_\frakp}}$ splits and $x$ has critical $\frakp$-slope'' forces the first order deformations of $\rho_x$ to have constant $\tau_\frakp$-Hodge--Tate--Sen weights.
As for other primes $\frakp'$ of $F$ above $p$ for which the condition does not hold, fixing Hodge--Tate--Sen weight is a codimension one condition.
Hence the number of $p$-adic primes for which the condition does not hold is basically the codimension of the tangent space of the fiber inside the whole tangent space of the eigenvariety.
This gives the lower bound of the dimension of the fiber as in the theorem.

We make a digression here to mention a few related works.
The analog of Theorem~\ref{thm:main} in the weight $1$ case has been studied by many people.
When $F=\QQ$, Bella\"iche--Dimitrov proved that a weight $1$ Hecke eigenform is ramified over the weight space if and only if it has real multiplication (RM) (\cite{BD}).
This is generalized by Betina to the case of Hilbert modular forms of parallel weight $1$ (\cite{Bet16}).
The analog of Theorem~\ref{thm:main} in the case of Eisenstein series of critical slope is also of interest.
When $F=\QQ$, Bella\"iche--Chenevier established an equivalent condition for an Eisenstein series of critical slope to be ramified over the weight space, in terms of $p$-adic zeta values (\cite{BC}).
There is also a forthcoming work of Adel Betina, Mladen Dimitrov and Sheng-Chi Shih, investigating the local structure of the Hilbert cuspidal eigenvariety at weight 1 Eisenstein points.

In the presence of a ramification point on the eigenvariety, one can further ask for an explicit description of the associated generalized Hecke eigenform.
Darmon--Lauder--Rotger gave a formula for the Fourier coefficients of the associated generalized Hecke eigenform in the weight $1$ RM case (\cite{DLR15}).
This is also generalized by Betina to Hilbert modular forms of parallel weight $1$ (\cite{Bet16}).

\addtocontents{toc}{\SkipTocEntry} 
\subsection*{Structure of the paper}
In Section~\ref{sec:Hilb} we review the various definitions of Hilbert modular forms.
In Section~\ref{sec:classicality} we define overconvergent Hilbert modular forms geometrically and recall the classicality theorem proven by Tian--Xiao.
We also prove the equivalence of (\ref{cond:geneigen}) and (\ref{cond:imtheta}) of Theorem~\ref{thm:main} in Corollary~\ref{cor:2and3}.
In Section~\ref{sec:Galrep} we begin the perspective of Galois representations and prove that (\ref{cond:imtheta}) implies (\ref{cond:split}) in Proposition~\ref{prop:3to4}.
In Section~\ref{sec:eigenvar} we review the construction of Hilbert cuspidal eigenvariety and deduce the equivalence of (\ref{cond:ram}) and (\ref{cond:geneigen}) in Lemma~\ref{lem:1and2}.
The last Section~\ref{sec:Galdeform} focuses on Galois deformation theory and we prove Theorem~\ref{thm:main2} (Theorem~\ref{thm:4to1}).

\addtocontents{toc}{\SkipTocEntry} 
\subsection*{Acknowledgments}
The author would like to thank her advisor Barry Mazur for his constant support.
She would also like to thank John Bergdall, Mark Kisin and Koji Shimizu for helpful discussions and comments, Mladen Dimitrov, Kai-Wen Lan and Cheng-Chiang Tsai for answering her questions, and Koji Shimizu and Mladen Dimitrov again for reading an early draft of the paper.
The author is partially supported by the Government Scholarship to Study Abroad from Taiwan.

\addtocontents{toc}{\SkipTocEntry} 
\subsection*{Notations} 
Fix a totally real field $F$ of degree $d$ over $\QQ$.
Let $\Sigma_\infty$ denote the set of archimedean places of $F$; in particular $\#\Sigma_\infty=d$.
Fix a rational prime $p$ which splits completely in $F$.
Let $\Sigma_p$ be the set of primes of $F$ above $p$, so $\#\Sigma_p=d$.
Fix an isomorphism $\iota_p\colon\CC\xrightarrow{\sim}\Qpbar$.
For each $\frakp\in\Sigma_p$,
denote by $\tau_\frakp\in\Sigma_\infty$ the archimedean place of $F$ such that $\iota_p\circ \tau_\frakp\colon F\rightarrow \Qpbar$ induces $\frakp$.

\section{Hilbert modular forms} \label{sec:Hilb}
Let $G$ be the algebraic group $\Res_{\cO_F/\ZZ}\GL_{2}$ over $\ZZ$, and denote by $Z$ the center of $G$.

\subsection{Weights} \label{subsec:weights}
Let $T$ be the diagonal subgroup, a maximal torus, of $G$.
Regarding $\Res_{\cO_F/\ZZ} \Gm$ as the diagonal subgroup of the derived group $G^\der = \Res_{\cO_F/\ZZ} \SL_{2}$ of $G$, we have a map 
$p_1 \colon \Res_{\cO_F/\ZZ} \Gm \rightarrow T$ given by $t \mapsto \diag(t,t^{-1})$. 
On the other hand, regarding $\Res_{\cO_F/\ZZ} \Gm$ as the center $Z$ of $G$, we have another map 
$p_2 \colon \Res_{\cO_F/\ZZ} \Gm \rightarrow T$ given by $t \mapsto \diag(t,t)$.
Then the map 
$p_1\times p_2 \colon \Res_{\cO_F/\ZZ} \Gm \times \Res_{\cO_F/\ZZ} \Gm \rightarrow T$ is surjective, and its kernel is $\Res_{\cO_F/\ZZ}\mu_2$, diagonally embedded into $\Res_{\cO_F/\ZZ}\Gm \times \Res_{\cO_F/\ZZ}\Gm$.
Hence the character group of $T$ is the subgroup of $\ZZ^{\Sigma_\infty}\times \ZZ^{\Sigma_\infty}$ consisting of those $(k_\tau, w_\tau)_\tau$ such that $k_\tau\congr w_\tau \pmod 2$ for all $\tau\in\Sigma_\infty$. 
We will only consider the subgroup of the character group of $T$ consisting of those characters which are of finite order when restricted to $Z(\ZZ)\isom\cO_F^\times$.
The condition means that $w_\tau=w$ is independent of $\tau\in\Sigma_\infty$.
This is because by the proof of Dirichlet unit theorem, the image of 
$\cO_F^\times \rightarrow \RR^{\Sigma_\infty}, a \mapsto (\log \left| \tau(a) \right| )_\tau$ forms a lattice inside the hyperplane $\sum_{\tau\in\Sigma_\infty} x_\tau = 0$.
We will see in Section~\ref{subsec:HilbMF} the reason why we only consider such characters of $T$.

By a \emph{weight}, we mean a tuple $(\underline{k},w)\in \ZZ^{\Sigma_\infty}\times \ZZ$ such that $w\congr k_\tau \pmod2$ for all $\tau\in\Sigma_\infty$. 
We say that a weight $(\underline{k},w)$ is \emph{cohomological} if $w\geq k_\tau\geq2$ for all $\tau\in\Sigma_\infty$.

\subsection{Automorphic perspective} \label{subsec:HilbMF}
We follow \cite[Chap.\ 3]{Bump} and \cite{BJ} for the exposition in this section.
Let $K_\infty$ be $\Or_2(F\tensor_\QQ \RR)$, a maximal compact subgroup of $G(\RR)$, and $K_\infty^0 = \SO_2(F\tensor_\QQ \RR)$ be its connected component containing the identity.
Let $\frakg = \frak{gl}_2(F\tensor_\QQ \RR)$ be the Lie algebra of $G(\RR)$, $U(\frakg)$ its universal enveloping algebra over $\CC$, and $Z(\frakg)$ the center of $U(\frakg)$. 
In our case, $Z(\frakg) \isom \CC[Z_\tau, \Delta_\tau]_{\tau\in\Sigma_\infty}$ (\cite[Sec.\ 3.2]{Bump}), 
where 
\begin{align*}
 Z_\tau&=\begin{pmatrix}1&\\&1\end{pmatrix}_\tau, \text{ and } \\
 \Delta_\tau&=-\frac14(H_\tau^2+2R_\tau L_\tau+2L_\tau R_\tau) \text{ is the Casimir element, with} \\
 R_\tau&=\frac12\begin{pmatrix}1&i\\i&-1\end{pmatrix}_\tau,
 L_\tau=\frac12\begin{pmatrix}1&-i\\ -i&-1\end{pmatrix}_\tau,
 H_\tau=-i\begin{pmatrix}&1\\ -1&\end{pmatrix}_\tau.
\end{align*}

An \emph{automorphic form for $G(\AA)$} is a function $\phi \colon G(\QQ) \backslash G(\AA) \rightarrow \CC$ satisfying the following conditions (\cite[Sec.\ 4]{BJ}):
\begin{enumerate}[label=(\arabic*)]
 \item $\phi$ is a smooth function, i.e., smooth in $g_\infty \in G(\RR)$ and locally constant in $g^\infty \in G(\AA^\infty)$.
 \item ($K$-finite) There exists a compact open subgroup $K \subset G(\AA^\infty)$ such that $\phi(gk)=\phi(g)$ for all $k\in K$.
 \item \label{Kfinite}
($K_\infty$-finite) $K_\infty \cdot \phi$ spans a finite dimensional subspace of $C^\infty(G(\QQ) \backslash G(\AA), \CC)$, where for $k_\infty \in K_\infty$, $(k_\infty\cdot \phi)(g) \defeq \phi(gk_\infty)$.
 \item \label{Zfinite}
($Z(\frakg)$-finite) There is an ideal $I$ of $Z(\frakg)$ of finite codimension annihilating $\phi$, where for $X \in Z(\frakg)$,
\[
 (X\cdot \phi)(g) \defeq \left.\frac{d}{dt}\phi\left(g \cdot \exp(tX)\right)\right|_{t=0}.
\]
 \item For each $g^\infty \in G(\AA^\infty)$, the function $\phi_{g^\infty} \colon G(\RR) \rightarrow \CC, g_\infty \mapsto \phi(g_\infty g^\infty)$ is slowly increasing, i.e., there exist a constant $C$ and a positive integer $N$ (depending on $g^\infty$) such that for all $g_\infty \in G(\RR)$,
\[
 \lvert \phi_{g^\infty}(g_\infty) \rvert \leq C \cdot {\lVert g_\infty \rVert}^N.
\]
Here the norm $\lVert g_\infty \rVert$ is the length of the vector $(\tau(g_\infty), \det \tau(g_\infty)^{-1})_\tau$ in the Euclidean space $\left( M_2(\RR) \directsum \RR \right)^{\Sigma_\infty} \isom \RR^{5d}$.
\end{enumerate}
The compact open subgroup $K$ in (2) is called the \emph{level} of $\phi$.

We say $\phi$ is \emph{cuspidal} if
\[
\int_{F\backslash\AA_F} \phi\left(\begin{pmatrix}1&u\\&1\end{pmatrix}g\right) du = 0 
\]
for all $g\in G(\AA)$.
Here $du$ is an additive Haar measure on $F\backslash\AA_F$.

We define the notion of \emph{weights} for automorphic forms for $G(\AA)$.
First note that $K_\infty \cdot (F\tensor_\QQ \RR)_{>0}$ is a maximal torus of $G(\RR)$, so according to the discussion in Section~\ref{subsec:weights}, its characters can be represented by $(k_\tau, w_\tau)_\tau \in \ZZ^{\Sigma_\infty} \times \ZZ^{\Sigma_\infty}$ with $k_\tau \congr w_\tau \pmod 2$.
Let $\frakh^\pm \defeq \PP^1(\CC) \setminus \PP^1(\RR)$ be the union of the upper and lower half planes in $\CC$.
Given $(k_\tau,w_\tau)_\tau$, define an automorphy factor $j \colon G(\RR)\times (\frakh^\pm)^{\Sigma_\infty} \rightarrow \CC$ by 
\[
 j(g_\infty,z) \defeq \prod_{\tau\in\Sigma_\infty} \det(g_\tau)^{-\frac{w_\tau+k_\tau-2}2}(c_\tau z_\tau +d_\tau)^{k_\tau},
\]
where we write $g_\infty = (g_\tau)_\tau = \begin{pmatrix}a_\tau&b_\tau\\c_\tau&d_\tau\end{pmatrix}_\tau\in G(\RR) = \prod_{\tau\in\Sigma_\infty} \GL_2(\RR)$.
An automorphic form $\phi$ for $G(\AA)$ is said to have \emph{weight $(k_\tau,w_\tau)_\tau$} if 
\begin{enumerate}[label=(\arabic*')]
\setcounter{enumi}{2}
 \item \label{Kfinite'}
for all $k_\infty \in K_\infty \cdot (F\tensor_\QQ \RR)_{>0}$, $\phi(g k_\infty) = j(k_\infty^{-1},(i,\ldots,i)) \phi(g)$, and
 \item \label{Zfinite'}
$\phi$ is annihilated by the ideal $I$ of $Z(\frakg)$ generated by $\Delta_\tau-\frac{k_\tau}2(1-\frac{k_\tau}2)$ and $Z_\tau-(w_\tau-2)$.
\end{enumerate}
Note that \ref{Kfinite'} implies \ref{Kfinite} and \ref{Zfinite'} implies \ref{Zfinite} in the definition of automorphic forms.

We show that the weights $(k_\tau, w_\tau)$ are always in the form $(\kunder,w)$, namely, $w_\tau=w$ is independent of $\tau$.
Let $\phi$ be an automorphic form for $G(\AA)$ of weight $(k_\tau,w_\tau)_\tau$ and level $K$.
We may assume that $\phi$ has a central character, i.e., there exists $\chi\colon Z(\QQ) \backslash Z(\AA) \rightarrow \CC^\times$ such that for all $z\in Z(\AA)$, $\phi(zg)=\chi(z)\phi(g)$.
This is because $K\subset Z(\AA^\infty)$ is of finite index in $Z(\AA^\infty)$, and hence $\phi$ is a finite sum of those automorphic forms of the same weight and level admitting central characters.
Write $\chi^\infty \colon Z(\AA^\infty) \rightarrow \CC^\times$ for the finite component of $\chi$.
Then for $z\in Z(\QQ)$, say $z = \begin{pmatrix}a&\\&a\end{pmatrix}$ with $a\in F^\times$, 
\[
 \phi(g) = \phi(zg) 
= j(z_\infty^{-1}, (i,\ldots,i)) \chi^\infty(z^\infty) \phi(g) 
= \prod_{\tau\in\Sigma_\infty} \tau(a)^{w_\tau-2}\chi^\infty(z^\infty)\phi(g).
\]
Since any character $\chi^\infty \colon Z(\AA^\infty) \rightarrow \CC^\times$ has finite order, so does $\prod_{\tau\in\Sigma_\infty} \tau(a)^{w_\tau-2}$ restricted to $Z(\ZZ) \isom \cO_F^\times$.
We have seen in Section~\ref{subsec:weights} that this condition implies $w_\tau=w$ is independent of $\tau\in\Sigma_\infty$.
This is why we defined a weight to be $(\kunder,w)$ instead of $(k_\tau,w_\tau)_\tau$.

We use automorphic forms for $G(\AA)$ to define Hilbert modular forms.
We have $(\frakh^\pm)^{\Sigma_\infty} = G(\RR)/K_\infty^0\cdot (F\tensor_\QQ \RR)_{>0}$. 
Let $K \subset G(\AA^\infty)$ be a compact open subgroup.
Given an automorphic form $\phi$ for $G(\AA)$ of weight $(\kunder,w)$ and level $K$, define 
\[
 f_\phi \colon (\frakh^\pm)^{\Sigma_\infty} \times G(\AA^\infty) \rightarrow \CC
\]
by 
\[
 f_\phi(g_\infty (i,\ldots,i),g^\infty) = j(g_\infty,(i,\ldots,i)) \phi(g_\infty g^\infty).\]
Since the stabilizer of $(i,\ldots,i)$ is $K_\infty^0 (F\tensor_\QQ \RR)_{>0}$, and $\phi$ satisfies the invariance property \ref{Kfinite'}, $f_\phi$ is well-defined.
In addition, $f_\phi$ satisfies the following conditions.
\begin{enumerate}
 \item $f_\phi(z,g^\infty)$ is holomorphic in $z$ and locally constant in $g^\infty$.
 \item For all $k\in K$, $f_\phi(z,g^\infty k)=f_\phi(z,g^\infty)$.
 \item For all $\gamma \in G(\QQ)$, 
\[
 f_\phi(\gamma z,\gamma g^\infty) = j(\gamma, z) f_\phi(z,g^\infty)
\]
 
\end{enumerate}
We call $f_\phi$ a \emph{Hilbert modular form of weight $(\kunder,w)$ and level $K$}.
We say $f_\phi$ is \emph{cuspidal} if $\phi$ is.
Write $M_{(\kunder,w)}(K,\CC)$ for the space of Hilbert modular forms of weight $(\kunder,w)$ and level $K$, 
and $S_{(\kunder,w)}(K,\CC)$ for the subspace of cusp forms.

\subsection{Geometric perspective} \label{sec:geom}
We follow \cite[Sec.\ 2]{TX16} and \cite[Chap.\ III]{Milne} for the exposition in this section.
Let $K\subset G(\AA^\infty)$ be a compact open subgroup.
Let $\bS \defeq \Res_{\CC/\RR}\Gm$ be the Deligne torus.
Let 
\[
 h_0 \colon \bS(\RR)\isom \CC^\times \rightarrow G(\RR)\isom \prod_{\tau\in\Sigma_\infty}\GL_2(\RR)
\]
be the homomorphism
\[
 a+bi \mapsto \left( \begin{pmatrix} a&-b\\ b&a\end{pmatrix}\right)_\tau.
\]
We may let $G(\RR)$ act on homomorphisms $h\colon \bS(\RR) \rightarrow G(\RR)$ by conjugation on the target.
The stabilizer of $h_0$ is $K_\infty^0 (F\tensor_\QQ \RR)_{>0}$.
Hence the $G(\RR)$-conjugacy class of $h_0$ is identified with the Hermitian symmetric domain $G(\RR)/K_\infty^0 (F\tensor_\QQ \RR)_{>0} = (\frakh^\pm)^{\Sigma_\infty}$.
Let $\Sh_K(G)$ be the Shimura variety of $G$ with level $K$; it is a complex algebraic variety with $\CC$-points
\[
 \Sh_K(G)(\CC) = G(\QQ) \backslash (\frakh^{\pm})^{\Sigma_\infty} \times G(\AA^\infty)/K.
\]

Let $\mu_0$ be the Hodge cocharacter associated to $h_0$, i.e., the homomorphism
\[
 \Gm(\CC) \xrightarrow{z \mapsto (z,1)} \bS(\CC) \xrightarrow{h_{0,\CC}} G(\CC),
\]
where by $(z,1)\in\bS(\CC)$ we mean $\bS(\CC)=\bS(\RR)\tensor_\RR \CC \isom \CC^\times\times\CC^\times$ through $z\tensor 1 \mapsto (z,\bar{z})$.
The \emph{compact dual} of $(\frakh^\pm)^{\Sigma_\infty}$ is the $G(\CC)$-conjugacy class of $\mu_0$.
In our case, it is $G(\CC)/P(\CC) = (\PP^1_\CC)^{\Sigma_\infty}$, where $P$ is the upper triangular subgroup, the standard Borel, of $G$.

We have the Borel embedding $\beta \colon (\frakh^\pm)^{\Sigma_\infty} \inj \PP^1(\CC)^{\Sigma_\infty}$ sending $h$ to its Hodge cocharacter $\mu_h$.
The compact dual $(\PP^1_\CC)^{\Sigma_\infty}$ has a natural $G_\CC$-action on it.
Given a $G_\CC$-bundle $\cJ$ on $(\PP^1_\CC)^{\Sigma_\infty}$, $\beta^{-1}(\cJ)$ is a $G(\RR)$-bundle on $(\frakh^\pm)^{\Sigma_\infty}$. 
Let $Z_s$ be the largest subtorus of $Z$ which splits over $\RR$ but has no subtorus splitting over $\QQ$.
Since $Z\isom \Res_{\cO_F/\ZZ}\Gm$, we have $Z_s \isom \ker[\Res_{\cO_F/\ZZ} \Gm \xrightarrow{\Nm_{\cO_F/\ZZ}} \Gm]$.
When the $G_\CC$-action on $\cJ$ is trivial on $Z_{s,\CC}$, one may define a vector bundle $\cV_K(\cJ)$ on $\Sh_K(G)$, 
\[
 \cV_K(\cJ) \defeq G(\QQ) \backslash \beta^{-1}(\cJ) \times G(\AA^\infty) / K
\]
when $K$ is sufficiently small.
Such vector bundles $\cV_K(\cJ)$ on $\Sh_K(G)$ are called \emph{automorphic vector bundles}.
Moreover, the category of $G_\CC$-vector bundles on the compact dual is equivalent to the category of finite-dimensional complex representations of $P_\CC$ (\cite[Remark III.2.3(a)]{Milne}).
   
Let $(\kunder, w)$ be a weight.
Let $\mathrm{Std}_{P,\tau}$ be the standard (1-dimensional) representation of $P_{\CC,\tau}$ on $\CC^2/(\CC^2)^{P_{\CC,\tau}}$, and $\check{\mathrm{Std}}_{P,\tau}$ its contragredient.
Let $\mathrm{det}_\tau$ be the $1$-dimensional representation of $P_{\CC,\tau}$ given by taking the determinant.
It can be computed that the automorphy factor of $\check{\mathrm{Std}}_{P,\tau}$ is 
\[
 G(\RR) \times (\frakh^\pm)^{\Sigma_\infty} \rightarrow \CC, (g_\infty, z) \mapsto (c_\tau z_\tau+d_\tau)\det(g_\tau)^{-1},
\]
where $g_\infty = (g_\tau)_\tau = \begin{pmatrix}a_\tau&b_\tau\\c_\tau&d_\tau\end{pmatrix}$.
Let $\omegaunder^{(\kunder,w)}$ be the automorphic vector bundle on $\Sh_K(G)$ coming from the ($1$-dimensional) representation 
\[
\Tensor_{\tau\in\Sigma_\infty} \left( \check{\mathrm{Std}}_{P,\tau}^{\tensor k_\tau} \tensor \det_\tau^{-\frac{w-k_\tau-2}2} \right)
\]
of $P_\CC$.
Note that $Z_{s,\CC}$ is in the kernel of the representation as long as the exponent $k_\tau$ of $\check{\mathrm{Std}}_{P,\tau}$ and the exponent $m_\tau$ of $\det_\tau$ are such that $-k_\tau + 2m_\tau$ is independent of $\tau$.
Then we have a geometric interpretation for the space of Hilbert modular forms
\[
 M_{(\kunder, w)}(K,\CC) = H^0(\Sh_K(G),\omegaunder^{(\kunder,w)}).
\]

There is a canonical model of $\Sh_K(G)$ over the reflex field $\QQ$ of $G$.
Assume the level $K$ is of the form $K^p K_p$ with $K^p\subset G(\AA^{\infty,p})$ and $K_p\subset G(\AA_p)$, and that $K_p$ is hyperspecial, i.e. $K_p = \GL_2(\cO_F \times \ZZ_p)$.
Then one can construct an integral model of $\Sh_K(G)$ over $\ZZ_{(p)}$ (\cite[Sec.\ 2.3]{TX16}, \cite[Sec.\ 3.1]{AIP16}).
Roughly speaking, one first constructs an integral model of $\Sh_K(G\times_{\Res_{\cO_F/\ZZ}} \Gm)$ as a disjoint union of $\cM_K^\frakc$, the moduli space of $\frakc$-polarized Hilbert-Blumenthal abelian varieties (HBAVs) of level $K$, where $\frakc$ is a fractional ideal of $\cO_F$ and runs through a fixed set of representatives for the strict class group $\operatorname{Cl}^+(F)$.
Then the integral model of $\Sh_K(G)$ over $\ZZ_{(p)}$ is the disjoint union of the quotient of $\cM_K^\frakc$ by $\det(K)\cap \cO_F^{\times,+}/(K\cap \cO_F^\times)^2$, where $\frakc$ again is a fractional ideal of $\cO_F$ and runs through a fixed set of representatives for the strict class group $\operatorname{Cl}^+(F)$ of $F$
Hence the integral model of $\Sh_K(G)$ is a course moduli space parametrizing HBAVs with extra data.
On the other hand, when 
\begin{align*} \label{cond:level} 
 \det(K)\cap \cO_F^{\times,+} = (K\cap \cO_F^\times)^2 \tag{*}, 
\end{align*}
the quotient map is in fact an isomorphism from a geometric component to its image, and hence $\Sh_K(G)$ is a fine moduli space in this case.
Moreover, given a compact open subgroup $K\subset G(\AA^\infty)$, by shrinking the prime-to-$p$ level $K^p$, one can always reach a level $K'$ for which (\ref{cond:level}) is satisfied (\cite[Lemma 2.5]{TX16}).
We continue to denote the integral model by $\Sh_K(G)$.

Let $K=K^p K_p$ with $K_p$ hyperspecial, and assume $K$ satisfies (\ref{cond:level}).
One can construct an arithmetic toroidal compactification $\Sh_K^\tor(G)$ of $\Sh_K(G)$ (\cite{Chai}\cite[Chap. 6]{Lan13}).
The arithmetic toroidal compactifications are smooth projective over $\ZZ_{(p)}$.
Let $\sfD$ be the toroidal boundary $\Sh_K^\tor(G)\setminus \Sh_K(G)$.
The toroidal boundary is a relative simple normal crossing Cartier divisor of $\Sh_K^\tor(G)$ relative to $\Spec \ZZ_{(p)}$.
There also exists a polarized semi-abelian scheme $\cA^{\sa}$ over $\Sh_K^\tor(G)$ with an $\cO_F$-action, extending the universal abelian scheme on $\Sh_K(G)$.
Let $\omegaunder$ be the pullback of $\Omega^1_{\cA^\sa/\Sh_K^\tor(G)}$ via the identity section.
This is an $(\cO_{\Sh_K^\tor(G)}\tensor_\ZZ \cO_F)$-module, locally free of rank $1$.
There exists a unique $(\cO_{\Sh_K^\tor(G)}\tensor_\ZZ\cO_F)$-module $\cH^1$, locally free of rank $2$, extending the relative first de Rham cohomology of the universal abelian scheme on $\Sh_K(G)$.
The Hodge filtration also extends:
\[
 0 \rightarrow \omegaunder \rightarrow \cH^1 \rightarrow \operatorname{Lie}((\cA^\sa)^\vee) \rightarrow 0.
\]
Let $F^\mathrm{Gal}$ be the Galois closure of $F$, and $R$ an $\cO_{F^\mathrm{Gal},(p)}$-module.
After base change to $R$, we can decompose the $\cO_F$-modules using the archimedean places $\tau \colon \cO_F \rightarrow R$ of $F$, and the $\tau$-component gives
\[
 0 \rightarrow \omegaunder_\tau \rightarrow \cH^1_\tau \rightarrow \operatorname{Lie}((\cA^\sa)^\vee)_\tau \rightarrow 0.
\]
Hence after base change to $R$, one can define an integral model of the automorphic vector bundles $\omegaunder^{(\kunder,w)}$ over $\Sh^\tor_K(G)_R$:
\[
 \omegaunder^{(\kunder,w)} \defeq \Tensor_{\tau\in\Sigma_\infty} \left(\omegaunder_\tau^{k_\tau} \tensor (\wedge^2 \cH^1_\tau)^{\frac{w-k_\tau-2}2}\right).
\]

Let $K=K^pK_p$ with $K_p$ hyperspecial, but not necessarily satisfying (\ref{cond:level}),
Define the space of \emph{Hilbert modular forms of weight $(\kunder,w)$ and level $K$ with coefficients in $R$} to be
\[
 M_{(\kunder,w)}(K,R) \defeq H^0(\Sh_{K'}^\tor(G)_R,\omegaunder^{(\kunder,w)})^{K/K'},
\]
where $K'\subset K$ is a compact open subgroup satisfying (\ref{cond:level}).
And we define the subspace of cuspidal Hilbert modular forms to be
\[
 S_{(\kunder,w)}(K,R) \defeq H^0(\Sh_{K'}^\tor(G)_R,\omegaunder^{(\kunder,w)}(-\sfD))^{K/K'}.
\]
By Koecher's principle, we have $M_{(\kunder,w)}(K,R)=H^0(\Sh_K(G)_R, \omegaunder^{(\kunder,w)})$, which coincides with the previous definition of Hilbert modular forms when $R=\CC$.

\subsection{Theta operators} \label{subsec:theta}
As before, we choose $T$, the diagonal subgroup of $G$, to be our fixed maximal torus, and $P$, the upper triangular subgroup, to be our fixed Borel subgroup of $G$.
The Weyl group $W$ of $G$ is $\{\pm1\}^{\Sigma_\infty}$.
For a subset $J\subset\Sigma_\infty$, let $s_J\in\{\pm1\}^{\Sigma_\infty}$ be the element whose $\tau$-component is $1$ if $\tau\in J$ and is $-1$ if $\tau\notin J$. 
In particular $s_{\Sigma_\infty}$ is the identity element.
We have the usual dot action of $W$ on the character group of $\Res_{\cO_F/\ZZ}$, the diagonal subgroup of $G^\der$: $w\cdot \chi = w(\chi+\rho)-\rho$, where $\rho$ is half of the sum of the positive roots.
In our notation, this means that $\{\pm1\}^{\Sigma_\infty}$ acts on $\ZZ^{\Sigma_\infty}$:
For $J\subset\Sigma_\infty$ and $\underline{k}\in\ZZ^{\Sigma_\infty}$, $s_J\cdot\underline{k}$ has $\tau$-component $k_\tau$ if $\tau\in J$ and $2-k_\tau$ if $\tau\notin J$.

Let $J$ be a subset of $\Sigma_\infty$.
Let $f$ be a local section of $\omegaunder^{(s_J\cdot\kunder,w)}$ with $q$-expansion at a cusp of $\Sh_K^\tor(G)_R$ being 
\[
 f=\sum_\xi a_\xi q^\xi,
\]
where $\xi$ runs through $0$ and the set of totally positive elements in a lattice of $F$.
Define a differential operator of order $k_\tau-1$
\[
 \Theta_{\tau,k_\tau-1}(f) = \frac{(-1)^{k_\tau-2}}{(k_\tau-2)!}\sum_\xi \tau(\xi)^{k_\tau-1}a_\xi q^\xi.
\]

\section{Classicality theorems} \label{sec:classicality}
\subsection{Overconvergent Hilbert modular forms} \label{subsec:OC}
Let $L$ be a subfield of $\CC$ containing $F^\mathrm{Gal}$.
The fixed embedding $\iota_p\colon\CC\xrightarrow{\sim}\Qpbar$ induces a $p$-adic place $\sP$ of $L$.
Let $k_0$ be the residue field of $L_\sP$.

As in Section~\ref{sec:geom}, let $\Sh_K(G)$ over $\ZZ_{(p)}$ be the integral model of the Shimura variety of $G$ of level $K$, where $K=K^pK_p$ with $K_p$ hyperspecial and $K$ satisfying (\ref{cond:level}).
To simplify notation, let $X_K$ and $X_K^\tor$ be $\Sh_K(G)$ and $\Sh_K^\tor(G)$ base changed from $\ZZ_{(p)}$ to the ring of Witt vectors $W(k_0)$, respectively.
Let $\Xbar_K$ and $\Xbar_K^\tor$ over $k_0$ be their special fibers.
Let $\frakX_K^\tor$ be the formal completion of $X_K^\tor$ along $\Xbar_K^\tor$.
Let $\cX_K^\tor$ be the rigid generic fiber of $\frakX_K^\tor$ base changed from $W(k_0)[1/p]$ to $L_\sP$.

For a locally closed subset $\overline{U}\subset \Xbar_K^\tor$, denote by $]\overline{U}[$ the inverse image of $\overline{U}$ under the specialization map $\cX_K^\tor\rightarrow\Xbar_K^\tor$. 
Let $\Xbar_K^{\tor,\ord}\subset \Xbar_K^\tor$ be the ordinary locus.
Let $j\colon ]\Xbar_K^{\tor,\ord}[ \inj \cX_K^\tor$ be the natural inclusion of rigid analytic spaces.
For a coherent sheaf $\cF$ on $\cX_K^\tor$, define $j^\dagger \cF$ to be the sheaf on $\cX_K^\tor$ such that, for all admissible open subset $U\subset \cX_K^\tor$ we have
\[
 \Gamma(U,j^\dagger\cF) \defeq \varinjlim_V \Gamma(V\cap U,\cF),
\]
where $V$ runs through a fundamental system of strict neighborhoods of $]\Xbar_K^{\tor,\ord}[$ in $\cX_K^\tor$.

Let $(\kunder,w)$ be a weight.
Assume $K$ does not necessarily satisfy (\ref{cond:level}), but choose $K'\subset K$ which does.
Define the space of \emph{overconvergent Hilbert modular forms of weight $(\kunder,w)$ and level $K$ with coefficients in $L_\sP$} to be
\[
 M^\dagger_{(\kunder,w)}(K,L_\sP) \defeq H^0(\cX_{K'}^\tor, j^\dagger \omegaunder^{(\kunder,w)})^{K/K'},
\]
and the subspace of overconvergent cuspidal Hilbert modular forms to be
\[
 S^\dagger_{(\kunder,w)}(K,L_\sP) \defeq H^0(\cX_{K'}^\tor, j^\dagger \omegaunder^{(\kunder,w)}(-\sfD))^{K/K'}.
\]

Note that the theta operator $\Theta_{\tau,k_\tau-1}$ defined in Section~\ref{subsec:theta} induces a map
\[
 \Theta_{\tau,k_\tau-1}\colon M^\dagger_{(s_{\Sigma_\infty\setminus\{\tau\}}\cdot \kunder,w)}(K,L_\sP)
\rightarrow
S^\dagger_{(\kunder,w)}(K,L_\sP).
\]
We write $\Theta_\kunder$ for the sum of these maps over $\tau\in\Sigma_\infty$
\[
 \Theta_\kunder \colon \Directsum_{\tau\in\Sigma_\infty} M^\dagger_{(s_{\Sigma_\infty\setminus\{\tau\}}\cdot \kunder,w)}(K,L_\sP)
\xrightarrow{\sum_{\tau\in\Sigma_\infty}\Theta_{\tau,k_\tau-1}} 
S^\dagger_{(\kunder,w)}(K,L_\sP).
\]

We recall the theory of canonical subgroups for Hilbert modular varieties (\cite{GK12}, \cite[Sec.\ 3.11]{TX16}), so that we may define \emph{classical} overconvergent Hilbert modular forms, as well as the partial Frobenius $\Fr_\frakp$ and $U_\frakp$-operator later in Section~\ref{subsec:FrU}.
Let $K=K^pK_p$ with $K_p$ hyperspecial.
Over the special fiber $\Xbar_K$, the Verschiebung $(\cA^\sa_{k_0})^{(p)} \rightarrow \cA^{\sa}_{k_0}$ induces an $\cO_F$-linear map on the $\cO_{\Xbar_K}$-modules of invariant differential $1$-forms
$\omegaunder \rightarrow \omegaunder^{(p)}$,
and hence a map
\[
 h_\tau \colon \omegaunder_\tau \rightarrow \omegaunder^p_{\sigma^{-1}\circ \tau}
\]
for each $\tau\in\Sigma_\infty$.
Here $\sigma$ is the Frobenius on $k_0$ and $\sigma$ acts on $\Sigma_\infty$ by using the fixed embedding $\iota_p\colon \CC \xrightarrow{\sim} \Qpbar$ to identify $\Sigma_\infty$ with the set of $p$-adic embeddings $\Hom(F,\Qpbar)$, and with $\Hom(\cO_F, k_0)$ since $p$ is unramified in $F$.
In other word, we have $h_\tau\in H^0(\Xbar_K, \omegaunder^p_{\sigma^{-1}\circ \tau}\tensor \omegaunder_\tau^{-1})$, called the \emph{partial Hasse invariant}.
Let $x\in \cX_K^{\tor}$.
For each $\tau\in\Sigma_\infty$, let $\widetilde{h}_\tau$ be a Zariski local lift of the partial Hasse invariant of $h_\tau$ at $x$.
Goren--Kassaei defined a tuple of numbers $\underline{\nu}(x) = (\nu_\tau(x))_{\tau\in\Sigma_\infty}\in \QQ^{\Sigma_\infty}$ by (\cite[Sec.\ 4.2]{GK12})
\[
 \nu_\tau(x) = \min\{ 1, \val_p\widetilde{h}_\tau(x) \}.
\]
For any $\underline{r}\in\QQ^{\Sigma_p}$, there is a strict neighborhood 
\[
 {]\Xbar^{\tor,\ord}[}_{\underline{r}} \defeq \{ x\in \Xbar^\tor \mid p\nu_{\sigma^{-1}\circ \tau_\frakp}(x) + \nu_{\tau_\frakp}(x) < r_\frakp, \forall \frakp\in\Sigma_p \}
\]
of $]\Xbar^{\tor,\ord}[$.
For $\underline{r}\in\QQ^{\Sigma_p}$ such that $r_\frakp<p$ for all $\frakp\in\Sigma_p$,
Goren--Kassaei proved that over $]\Xbar^{\tor,\ord}[_{\underline{r}}$
there exists an $\cO_F$-invariant finite flat subgroup scheme $\cC_\frakp\subset \cA^\sa[\frakp]$ \'etale locally isomorphic to $\cO_F/\frakp$, called the \emph{universal $\frakp$-canonical subgroup} (\cite[Theorem A.1.3]{GK12}), extending the multiplicative part of $\cA^\sa[\frakp]$.

Denote the Iwahori subgroup of $\GL_2(\cO_F\tensor_\ZZ \ZZ_p)$ by $\Iw_p$. 
Thus $\Iw_p = \prod_{\frakp\in \Sigma_p}\Iw_\frakp$, where 
\[
 \Iw_\frakp = \left\{ \begin{pmatrix}a&b\\c&d\end{pmatrix} \in \GL_2(\cO_{F_\frakp}) \mid c\congr 0 \bmod \frakp \right\}.
\]
\begin{lem} \label{lem:inj}
 There are canonical injections
\begin{align*}
 \iota\colon M_{(\kunder,w)}(K^p\Iw_p,L_\sP) & \inj M^\dagger_{(\kunder,w)}(K,L_\sP) \\
 \iota^{\textrm{cusp}}\colon S_{(\kunder,w)}(K^p\Iw_p,L_\sP) &\inj S^\dagger_{(\kunder,w)}(K,L_\sP)  
\end{align*}
\end{lem}
\begin{proof}
For $\underline{r}\in\QQ^{\Sigma_p}$ such that $r_\frakp<p$ for all $\frakp\in\Sigma_p$,
the existence of canonical subgroups over ${]\Xbar_K^{\tor,\ord}[}_{\underline{r}}$ says that we have a section 
\[
 {]\Xbar_K^{\tor,\ord}[}_{\underline{r}} \inj \cX_{K^p\Iw_p}^\tor
\]
of the natural projection $\cX_{K^p\Iw_p}^\tor \rightarrow \cX_{K}^\tor$.
Hence pullback gives 
\[
 H^0(\cX_{K^p\Iw_p}^\tor,\omegaunder^{(\kunder,w)}) \rightarrow H^0({]\Xbar_K^{\tor,\ord}[}_{\underline{r}}, \omegaunder^{(\kunder,w)}).
\]
This is an injection because the zero set of any section of $\omegaunder^{(\kunder,w)}$ must have positive codimension.
The left hand side is exactly $M_{(\kunder,w)}(K^p\Iw_p,L_\sP)$ by rigid GAGA.
The right hand side admits a map into $\varinjlim_V H^0(V,\omegaunder^{(\kunder,w)}) = H^0(\cX_K^\tor, j^\dagger\omegaunder^{(\kunder,w)})$.
Hence we obtain the desired injection
\[
 \iota\colon M_{(\kunder,w)}(K^p\Iw_p,L_\sP) \inj M^\dagger_{(\kunder,w)}(K,L_\sP).
\]
The same proof works for the cuspidal case.
\end{proof}

\begin{defn} \label{defn:classical}
 An overconvergent cuspidal Hilbert modular form $f\in S^\dagger_{(\kunder,w)}(K,L_\sP)$ is called \emph{classical} if it lies in the image of $\iota^\textrm{cusp}$.
\end{defn}

\subsection{Rigid Cohomology}
Let $(\kunder,w)$ be a cohomological weight.
Let $\mathrm{Std}_{G,\tau} \colon G_\CC\isom(\GL_{2,\CC})^{\Sigma_\infty}\rightarrow \GL_{2,\CC}$ be the standard representation of $G$ of projection onto the $\tau$-factor, and $\check{\mathrm{Std}}_{G,\tau}$ its contragredient.
Write $\sF^{(\kunder,w)}$ for the automorphic vector bundle on $X_{K,\CC}$ coming from the representation
\[
\Tensor_{\tau\in\Sigma_\infty}\left(\Sym^{k_\tau-2}(\check{\mathrm{Std}}_{G,\tau})\tensor \det_\tau^{-\frac{w-k_\tau}2}\right)
\]
of $G_\CC$ (and in particular of $P_\CC$), which is trivial on $Z_{s,\CC}$.
The fact that $\sF^{(\kunder,w)}$ comes from a representation of $G_\CC$ and not just $P_\CC$ gives an integrable connection 
\[
 \nabla\colon \sF^{(\kunder,w)}\rightarrow \sF^{(\kunder,w)}\tensor \Omega^1_{X_{K,\CC}}.
\]
on $\sF^{(\kunder,w)}$ (\cite[Remark III.2.3(b)]{Milne}).

Let $R$ be an $\cO_{F^\mathrm{Gal},(p)}$-algebra.
Then we can define an integral model of $(\sF^{(\kunder,w)}, \nabla)$ on $X_{K,R}^\tor$
\[
 \sF^{(\kunder,w)} \defeq \Tensor_{\tau\in\Sigma_\infty} \left( \Sym^{k_\tau-2} \cH^1_\tau \tensor (\wedge^2\cH^1_\tau)^{\frac{w-k_\tau}2} \right).
\]
Note that 
\[
 \nabla\colon \sF^{(\kunder,w)}\rightarrow \sF^{(\kunder,w)}\tensor \Omega^1_{X_{K,R}^\tor}(\log \sfD).
\]
has log poles along $\sfD$.

Denote by $\DR^\bullet(\sF^{(\kunder,w)})$ the de Rham complex
\[
 \sF^{(\kunder,w)} \xrightarrow{\nabla} 
 \sF^{(\kunder,w)}\tensor \Omega^1_{X_{K,R}^\tor}(\log \sfD) 
 \xrightarrow{\nabla} \cdots \xrightarrow{\nabla} 
 \sF^{(\kunder,w)}\tensor \Omega^d_{X_{K,R}^\tor}(\log \sfD). 
\]
and by $\DR_c^\bullet(\sF^{(\kunder,w)})$ the complex $\DR^\bullet(\sF^{(\kunder,w)})$ tensored with $\cO_{\Sh_K^\tor(G)_R}(-\sfD)$.

Define
\[
 R\Gamma_\rig(\Xbar_K^{\tor,\ord},\sF^{(\kunder,w)})
 \defeq R\Gamma(\cX_K^\tor,j^\dagger \DR^\bullet(\sF^{(\kunder,w)})),
\]
and the compactly-supported version
\[
R\Gamma_{\rig,c}(\Xbar_K^{\tor,\ord},\sF^{(\kunder,w)})
\defeq R\Gamma(\cX_K^\tor,j^\dagger \DR_c^\bullet(\sF^{(\kunder,w)}))
\]
as objects in the derived category of $L$-vector spaces.
The their cohomology groups are
\[
 H^\star_\rig(\Xbar_K^{\tor,\ord},\sF^{(\kunder,w)}) \defeq \HH^\star(\cX_K^\tor,j^\dagger\DR^\bullet(\sF^{(\kunder,w)}))
\]
and 
\[
 H^\star_{\rig,c}(\Xbar_K^{\tor,\ord},\sF^{(\kunder,w)}) \defeq \HH^\star(\cX_K^\tor,j^\dagger\DR_c^\bullet(\sF^{(\kunder,w)}))
\]
where $\HH^\bullet$ denotes the hypercohomologies of the (compactly-supported) overconvergent de Rham complex.

\subsection{Classicality} \label{subsec:cl}
We recall the classicality results of Tian and Xiao (\cite{TX16}).
Let 
\[
\cH(K^p, L_\sP) = L_\sP[K^p \backslash G(\AA^{\infty,p}) /K^p]
\]
be the abstract prime-to-$p$ Hecke algebra of level $K^p$.
On $M^\dagger_{(\kunder,w)}(K,L_\sP)$ and $H^\ast_\rig(\Xbar_K^{\tor,\ord}, \sF^{(\kunder,w)})$, one may define the action of $\cH(K^p, L_\sP)$, $S_\frakp$, the partial Frobenius $\Fr_\frakp$ and the $U_\frakp$-operator (\cite[3.7--3.20]{TX16}).

One important ingredient of the classicality theorem is the following.
\begin{thm}[{\cite[Theorem 3.5]{TX16}}] \label{thm:OCandRigid}
There are isomorphisms of $L_\sP$-vector spaces
\[
 M^\dagger_{(\underline{k},w)}(K,L_\sP) / \Theta_\kunder \left(\Directsum_{\tau\in\Sigma_\infty}M^\dagger_{(s_{\Sigma_\infty\setminus\{\tau\}}\cdot \underline{k},w)}(K,L_\sP)\right)
\isom
H^d_\rig(\Xbar_K^{\tor,\ord}, \mathscr{F}^{(\underline{k},w)})\tensor_L L_\sP,
\]
and
\[
 S^\dagger_{(\underline{k},w)}(K,L_\sP) / \Theta_\kunder \left(\Directsum_{\tau\in\Sigma_\infty}S^\dagger_{(s_{\Sigma_\infty\setminus\{\tau\}}\cdot \underline{k},w)}(K,L_\sP)\right)
\isom
H^d_{\rig,c}(\Xbar_K^{\tor,\ord}, \mathscr{F}^{(\underline{k},w)})\tensor_L L_\sP
\]
equivariant under the actions of $\cH(K^p,L_\sP)$, $S_\frakp$, $\Fr_\frakp$ and $U_\frakp$.
\end{thm}

Let $\frakN$ be an integral ideal of $\cO_F$ prime to $p$.
Put
\[
 K_1(\frakN) = \left\{\begin{pmatrix}a&b\\c&d\end{pmatrix}\in \GL_2(\hat\cO_F) \mid a\congr1, c\congr0 \bmod \frakN \right\},
\]
and $K_1(\frakN)^p$ its prime-to-$p$ part.

\begin{thm}[{\cite[Theorem 6.9]{TX16}}] \label{thm:strongcl}
 Let $f\in S^\dagger_{(\kunder,w)}(K_1(\frakN), L_\sP)$ be an overconvergent cuspidal Hilbert Hecke eigenform.
For each prime $\frakp\in\Sigma_p$, let $\lambda_\frakp$ be the $U_\frakp$-eigenvalue of $f$.
 If 
 \[
  \val_p(\lambda_\frakp)<\frac{w+k_{\tau_\frakp}-2}2
 \]
 for each $\frakp$, then $f$ lies in $S_{(\kunder,w)}(K_1(\frakN)^p\Iw_p,L_\sP)$.
\end{thm}

\subsection{Definition of $\Fr_\frakp$, $U_\frakp$ and $S_\frakp$} \label{subsec:FrU}
We will need to compute with $\Fr_\frakp$, $U_\frakp$ and $S_\frakp$, so we recall how their actions are defined on the rigid cohomology (\cite[3.12-3.18]{TX16}).
Recall that given $\underline{r}\in\QQ^{\Sigma_p}$ with $r_\frakp<p$ for all $\frakp\in\Sigma_p$, we have the universal $\frakp$-canonical subgroup $\cC_\frakp\subset \cA^{\sa}[\frakp]$ over the strict neighborhood $]\Xbar_K^{\tor,\ord}[_{\underline{p}}$ of $]\Xbar_K^{\tor,\ord}[$ 
(Section~\ref{subsec:OC}).

Let $]\Xbar_K^{\tor,\ord}[_{\underline{r}}^{\frakp\mathrm{-can}}$ be the rigid analytic space $]\Xbar_K^{\tor,\ord}[_{\underline{r}}$, but regarded as classifying $\cA^{\sa}$ along with its $\frakp$-canonical subgroup $\cC_\frakp$.
Let $\underline{r}^\frakp\in\QQ^{\Sigma_p}$ be such that $r_\frakq^\frakp = r_\frakq$ for $\frakq \neq \frakp\in\Sigma_\frakp$ and $r_\frakp^\frakp = p \cdot r_\frakp$.
If $r_\frakq^\frakp<p$ for all $\frakq\in\Sigma_p$,
then we have 
\begin{center}
 \begin{tikzcd}
  & ]\Xbar_K^{\tor,\ord}[_{\underline{r}}^{\frakp\mathrm{-can}} \ar[ld,"p_1"'] \ar[rd,"p_2"] & \\
  ]\Xbar_K^{\tor,\ord}[_{\underline{r}} & & ]\Xbar_K^{\tor,\ord}[_{\underline{r}^\frakp}
 \end{tikzcd}
\end{center}
Here $p_1$ is given by $(\cA^{\sa}, \cC_\frakp) \mapsto \cA^{\sa}$ and $p_2$ is given by $(\cA^{\sa}, \cC_\frakp) \mapsto \cA^{\sa}/\cC_\frakp$ (\cite[Theorem 5.4.4(1), Appendix]{GK12}).
Note that $p_1$ is an isomorphism.
We have an isogeny 
\[
 \pi_\frakp \colon \cA^\sa \isom p_1^\ast\cA^{\sa} \rightarrow p_1^\ast\cA^{\sa}/\cC_\frakp =  p_2^\ast\cA^{\sa}.
\]

Let $]\Xbar_K^{\tor,\ord}[_{\underline{r}^\frakp}^{\frakp\mathrm{-nc}}$ be the rigid analytic space over $]\Xbar_K^{\tor,\ord}[_{\underline{r}^\frakp}$ classifying $\cO_F$-invariant finite flat subgroup schemes $\cD \subset \cA^\sa[\frakp]$ \'etale locally isomorphic to $\cO_F/\frakp$ and which are not the $\frakp$-canonical subgroup $\cC_\frakp$.
Then we have
\begin{center}
 \begin{tikzcd}
  & ]\Xbar_K^{\tor,\ord}[_{\underline{r}^{\frakp}}^{\frakp\mathrm{-nc}} \ar[ld,"q_1"'] \ar[rd,"q_2"] & \\
  ]\Xbar_K^{\tor,\ord}[_{\underline{r}^\frakp} & & ]\Xbar_K^{\tor,\ord}[_{\underline{r}}
 \end{tikzcd}
\end{center}
Here $q_1$ is given by $(\cA^{\sa}, \cD) \mapsto \cA^{\sa}$ and $q_2$ is given by $(\cA^{\sa}, \cD) \mapsto \cA^{\sa}/\cD$ (\cite[Theorem 5.4.4(4), Appendix]{GK12}).
Note that $q_2$ is an isomorphism.
We have an isogeny 
\[
 \check\pi_\frakp \colon q_1^\ast\cA^{\sa} \rightarrow q_1^\ast\cA^{\sa}/\cD =  q_2^\ast\cA^{\sa}.
\]

Let $\varpi_\frakp$ be an uniformizer of $\frakp$.
Let $[\varpi_\frakp]$ be the isogeny of $\cA^{\sa}$ given by multiplication by $\varpi_\frakp$.
The isogeny induces an automorphism
\[
 r \colon ]\Xbar_K^{\tor,\ord}[_{\underline{r}}\rightarrow ]\Xbar_K^{\tor,\ord}[_{\underline{r}}.
\]
We also let $\widetilde{r}$ be the automorphism induced by $[\varpi_\frakp]$ on $]\Xbar_K^{\tor,\ord}[^{\frakp\text{-nc}}_{\underline{r}^\frakp}$.
Hence we have
\[
 [\varpi_\frakp] \colon \cA^{\sa} \rightarrow r^\ast \cA^{\sa}.
\]
Since the canonical subgroup of $q_2^\ast\cA^{\sa}$ is $q_1^\ast\cA^{\sa}[\frakp]/\cD$, we have $q_2^\ast\pi_\frakp \circ \check\pi_\frakp = [\varpi_\frakp]$ as well as  $p_2 \circ p_1^{-1} \circ q_2 = q_1\circ \widetilde{r}$.

The isogeny $\pi_\frakp$ induces a map on the relative first de Rham cohomologies
\[
 \pi_\frakp^\ast \colon p_2^\ast \cH^1 \rightarrow p_1^\ast \cH^1 \isom \cH^1
\]
and hence on the automorphic vector bundles
\[
 \pi_\frakp^\ast \colon p_2^\ast \sF^{(\kunder,w)} \rightarrow \sF^{(\kunder,w)}
\]
and on the de Rham complexes
\[
 \pi_\frakp^\ast \colon \DR^\bullet(p_2^\ast \sF^{(\kunder,w)}) \rightarrow \DR^\bullet(\sF^{(\kunder,w)}).
\]
We then define the partial Frobenius $\Fr_\frakp$ as ${p_1}_\ast \circ \pi_\frakp^\ast \circ p_2^\ast$, or more precisely
\begin{center}
\scriptsize
 \begin{tikzcd}[column sep=small]
 R\Gamma(]\Xbar_K^{\tor,\ord}[_{\underline{r}^\frakp}, \DR^\bullet(\sF^{(\kunder,w)}))   \ar[d, "p_2^\ast"'] \ar[rr, dashrightarrow, "\Fr_\frakp"]
&  
& R\Gamma(]\Xbar_K^{\tor,\ord}[_{\underline{r}}, \DR^\bullet(\sF^{(\kunder,w)}))
\\
R\Gamma(]\Xbar_K^{\tor,\ord}[_{\underline{r}}^{\frakp\mathrm{-can}}, \DR^\bullet(p_2^\ast \sF^{(\kunder,w)}))
\ar[r,"\pi_\frakp^\ast"]
& R\Gamma(]\Xbar_K^{\tor,\ord}[_{\underline{r}}^{\frakp\mathrm{-can}}, \DR^\bullet(p_1^\ast \sF^{(\kunder,w)})) \ar[r,equal]
&
R\Gamma(]\Xbar_K^{\tor,\ord}[_{\underline{r}}, \DR^\bullet({p_1}_\ast p_1^\ast \sF^{(\kunder,w)})) \ar[u,"{p_1}_\ast"'],
 \end{tikzcd}
\end{center}
where $p_2^\ast$ comes from the adjunction map $\DR^\bullet(\sF^{(\kunder,w)}) \rightarrow \DR^\bullet({p_2}_\ast p_2^\ast \sF^{(\kunder,w)})$, and $p_1^\ast$ is the trace map $\DR^\bullet({p_1}_\ast p_1^\ast \sF^{(\kunder,w)}) \rightarrow \DR^\bullet(\sF^{(\kunder,w)})$ of the finite flat $p_1$ (in fact $p_1$ is an isomorphism).
Taking $\varinjlim$ over $\underline{r}$ and taking cohomology, we then get $\Fr_\frakp$ on the rigid cohomology
\[
 \Fr_\frakp \colon H^\star_\rig(\Xbar_K^{\tor,\ord},\sF^{(\kunder,w)}) \rightarrow H^\star_\rig(\Xbar_K^{\tor,\ord},\sF^{(\kunder,w)}).
\]

Similarly, we define the $U_\frakp$-operator as ${q_1}_\ast \circ \check\pi_\frakp^\ast \circ q_2^\ast$ on $H^\star_\rig(\Xbar_K^{\tor,\ord},\sF^{(\kunder,w)})$.
We also define the $S_\frakp$-operator as $[\varpi_\frakp]^\ast \circ r^\ast$ on $H^\star_\rig(\Xbar_K^{\tor,\ord},\sF^{(\kunder,w)})$.

\subsection{The adjoint of $U_\frakp$}
Let $H^d_{\rig,!}(\Xbar^{\tor,\ord}_{K_1(\frakN)},\sF^{(\kunder,w)})$ be the interior cohomology, i.e., 
\[
 H^d_{\rig,!}(\Xbar^{\tor,\ord}_{K_1(\frakN)},\sF^{(\kunder,w)}) \defeq
\im \left[
H^d_{\rig,c}(\Xbar^{\tor,\ord}_{K_1(\frakN)},\sF^{(\kunder,w)}) \rightarrow
H^d_{\rig}(\Xbar^{\tor,\ord}_{K_1(\frakN)},\sF^{(\kunder,w)}) \right].
\]
The Poincar\'e duality
\[
 H^d_{\rig}(\Xbar_{K_1(\frakN)}^{\tor,\ord}, \mathscr{F}^{(\underline{k},w)}) 
\times 
H^d_{\rig,c}(\Xbar_{K_1(\frakN)}^{\tor,\ord}, \check{\mathscr{F}}^{(\underline{k},w)}) 
\rightarrow 
H^{2d}_{\rig,c}(\Xbar_{K_1(\frakN)}^{\tor,\ord}, L)
\isom L,
\]
induces a perfect pairing on the interior cohomologies.
One can twist the second factor by the Atkin--Lehner operator $w_\frakN$ to make the pairing equivariant for Hecke operators $T_\frakl$ and $S_\frakl$ where $\frakl$ is a prime of $F$ not dividing $p\frakN$, as well as $S_\frakp$ (\cite[Sec.\ 3.9]{Dimit}).
More precisely, let $\ast$ be the involution $g\mapsto \det(g)^{-1}g$ on $G=\Res_{\cO_F/\ZZ}\GL_2$, which induces an isomorphism
\[
 \ast \colon H^d_{\rig,!}(\Xbar_{K_1(\frakN)^\ast}^{\tor,\ord}, \sF^{(\kunder,w)}) \xrightarrow{\sim} H^d_{\rig,!}(\Xbar_{K_1(\frakN)}^{\tor,\ord}, \check\sF^{(\kunder,w)}).
\]
Let $w_\frakN \in G(\ZZ)=\GL_2(\cO_F)$ be a matrix in $\begin{pmatrix} & -1\\ \frakN &\end{pmatrix}$.
Then $w_\frakN K_1(\frakN)^\ast w_\frakN^{-1} = K_1(\frakN)$, and thus 
\[
 w_\frakN \colon \Xbar_{K_1(\frakN)^\ast}^{\tor,\ord} 
\rightarrow \Xbar_{w_\frakN^{-1} K_1(\frakN)^\ast w_\frakN}^{\tor,\ord} = \Xbar_{K_1(\frakN)}^{\tor,\ord}.
\]
This induces an isomorphism on cohomologies
\[
 w_\frakN \colon  H^d_{\rig,!}(\Xbar_{K_1(\frakN)}^{\tor,\ord}, \sF^{(\kunder,w)}) \xrightarrow{\sim}  H^d_{\rig,!}(\Xbar_{K_1(\frakN)^\ast}^{\tor,\ord}, \sF^{(\kunder,w)}). 
\]
Applying the isomorphism $(*\circ w_\frakN)^{-1}$ on the second factor, the coefficient sheaf $\check\sF^{(\kunder,w)}$ is dualized back to $\sF^{(\kunder,w)}$ and the Poincar\'e pairing becomes Hecke equivariant.
We denote the modified Poincar\'e pairing by 
\[
 (\;,\;) \colon 
H^d_{\rig,!}(\Xbar_{K_1(\frakN)}^{\tor,\ord}, \mathscr{F}^{(\underline{k},w)})
\times 
H^d_{\rig,!}(\Xbar_{K_1(\frakN)}^{\tor,\ord}, \mathscr{F}^{(\underline{k},w)}) 
\rightarrow 
L.
\]
We want to compute the adjoint operator of $U_\frakp$ with respect to the modified Poincar\'e pairing.

\begin{lem} \label{lem:Upadj}
 The adjoint operator of $U_\frakp$ with respect to $(\;,\;)$ is $\Fr_\frakp$.
\end{lem}
\begin{proof}

The $U_\frakp$-operator on $H^d_{\rig, !}(\Xbar_{K_1(\frakN)}^{\tor,\ord}, \sF^{(\kunder,w)})$ is defined through the $U_\frakp$-correspondence using
\begin{align*}
 q_2^\ast &\colon \sF^{(\kunder,w)} \rightarrow {q_2}_\ast q_2^\ast \sF^{(\kunder,w)} \\
 \check\pi_\frakp^\ast &\colon q_2^\ast\sF^{(\kunder,w)} \rightarrow q_1^\ast \sF^{(\kunder,w)}\\
 {q_1}_\ast &\colon {q_1}_\ast q_1^\ast \sF^{(\kunder,w)} \rightarrow \sF^{(\kunder,w)}.
\end{align*}
Hence on $H^d_{\rig, !}(\Xbar_{K_1(\frakN)}^{\tor,\ord}, \check \sF^{(\kunder,w)})$, the adjoint operator of $U_\frakp$ with respect to the (non-twisted) Poincar\'e pairing is defined through the transpose correspondence, i.e. using
\begin{align*}
 q_1^\ast &\colon \check\sF^{(\kunder,w)} \rightarrow {q_1}_\ast q_1^\ast \check\sF^{(\kunder,w)} \\
 (\check\pi_\frakp^\ast)^\vee &\colon q_1^\ast\check\sF^{(\kunder,w)} \rightarrow q_2^\ast \check\sF^{(\kunder,w)}\\
 {q_2}_\ast &\colon {q_2}_\ast q_2^\ast \check\sF^{(\kunder,w)} \rightarrow \check\sF^{(\kunder,w)}.
\end{align*}

We can write $(\check\pi_\frakp^\ast)^\vee$ in another way:
We have seen that $p_2 \circ p_1^{-1} \circ q_2 = q_1\circ \widetilde{r}$.
Also note that $\check \sF^{(\kunder,w)} = \sF^{(\kunder,-w+4)}$.
Hence we have the map
\[
 q_1^\ast \sF^{(\kunder,-w+4)} 
\xrightarrow[\sim]{\widetilde{r}^\ast} \widetilde r^\ast \circ q_1^\ast \sF^{(\kunder,-w+4)} 
= q_2^\ast \circ {p_1^{-1}}^\ast \circ p_2^\ast \sF^{(\kunder,-w+4)} \xrightarrow{\pi_\frakp^\ast} q_2^\ast \sF^{(\kunder,-w+4)}.
\]
Then in fact $(\check\pi_\frakp^\ast)^\vee = \Nm_{F/\QQ}(\frakp)^{-(w-2)} S_\frakp^{-1} \pi_\frakp^\ast$.
This is because $\check\pi_\frakp^\ast \circ (\check\pi_\frakp^\ast)^\vee$ is multiplication by $\Nm_{F/\QQ}(\frakp)^{(-w+4)-2}$ and $\check\pi_\frakp^\ast \circ \pi_\frakp^\ast = [\varpi_\frakp]^\ast \circ \widetilde{r}^\ast$ is the $S_\frakp$-operator.

Hence we conclude that the adjoint operator of $U_\frakp$ with respect to the (non-twisted) Poincar\'e pairing is 
\begin{align*}
  (\Nm_{F/\QQ}(\frakp)^{-(w-2)}S_\frakp^{-1}) \cdot {q_2}_\ast \circ \pi_\frakp^\ast \circ \widetilde{r}^\ast \circ q_1^\ast 
&= (\Nm_{F/\QQ}(\frakp)^{-(w-2)}S_\frakp^{-1}) \cdot {q_2}_\ast \circ \pi_\frakp^\ast \circ q_2^\ast \circ {p_1^{-1}}^\ast \circ p_2^\ast \\
&= (\Nm_{F/\QQ}(\frakp)^{-(w-2)}S_\frakp^{-1}) \cdot {q_2}_\ast \circ q_2^\ast \circ {p_1^{-1}}^\ast \circ \pi_\frakp^\ast \circ p_2^\ast \\
&= (\Nm_{F/\QQ}(\frakp)^{-(w-2)}S_\frakp^{-1}) \cdot {q_2}_\ast \circ q_2^\ast \circ {p_1}_\ast \circ \pi_\frakp^\ast \circ p_2^\ast \\
&= (\Nm_{F/\QQ}(\frakp)^{-(w-2)}S_\frakp^{-1}) \cdot {p_1}_\ast \circ \pi_\frakp^\ast \circ p_2^\ast \\
&=(\Nm_{F/\QQ}(\frakp)^{-(w-2)}S_\frakp^{-1}) \cdot \Fr_\frakp
\end{align*}
Here the second equality is because the morphism $\pi_\frakp^\ast$ induced by isogeny commutes with base change, and the fourth equality is because $q_2$ is an isomorphism and hence ${q_2}_\ast \circ q_2^\ast = \id$.

Now we consider the modified Poincar\'e pairing.
As in \cite[Sec.\ 3.9]{Dimit}, $S_\frakp^{-1}$ is transformed into $S_\frakp$ under conjugation by $w_\frakN\circ\ast$.
Fix an $N=\Nm_{F/\QQ}(\frakN)$-th root of unity $\zeta_N$.
The $K_1(\frakN)$-level parametrizes an $\frakN$-torsion point $P$ of a polarized HBAV $A$, and the Atkin-Lehner operator $w_\frakN$ maps $(A,P)$ to $(A/(P), Q)$, where $Q\in A[\frakN]$ paired with $P$ is mapped to $\zeta_N$ under the Weil pairing. 
We then have $\Fr_\frakp w_\frakN = S_\frakp^{-1} \Nm_{F/\QQ}(\frakp)^{w-2} w_\frakN \Fr_\frakp$ (\cite[Intro.\ 8.II]{MW}).
This follows from compatibility of the Weil pairing with isogenies and the fact that $S_\frakp^{-1} \Nm_{F/\QQ}(\frakp)^{w-2}$ is the diamond operator at $\frakp$, which maps $(A,P)$ to $(A,\varpi_\frakp \cdot P)$.
We conclude that the adjoint operator of $U_\frakp$ with respect to the modified Poincar\'e pairing is 
\begin{align*}
(\ast \circ w_\frakN)^{-1} \circ \left[ (\Nm_{F/\QQ}(\frakp)^{-(w-2)}S_\frakp^{-1}) \cdot \Fr_\frakp \right] \circ (\ast \circ w_\frakN)
&= \Nm_{F/\QQ}(\frakp)^{-(w-2)} S_\frakp \cdot S_\frakp^{-1} \Nm_{F/\QQ}(\frakp)^{w-2} \Fr_\frakp \\
&=\Fr_\frakp.
\end{align*}
\end{proof}

\subsection{Classicality in critical slope}
We deduce from Section~\ref{subsec:cl} some classicality results in the case $\val_p(\lambda_\frakp) = \frac{w+k_{\tau_\frakp}-2}2$.
\begin{defn}
 Let $M$ be an $L_\sP$-vector space on which $U_\frakp$ acts for all $\frakp\in\Sigma_p$, and $\alphaunder\in \QQ^{\Sigma_p}$.
 We write $M_\alphaunder$ for the \emph{slope $\alphaunder$ part of $M$.}
 Namely, $M_\alphaunder$ is the sub-$L_\sP$-vector space of $M$ consisting of $m\in M$ such that for all $\frakp\in\Sigma_p$, there exists a polynomial $P_\frakp(T)\in L_\sP[T]$ such that its roots in $\CC_p$ all have $p$-adic valuation $\alpha_\frakp$ and $P_\frakp(U_\frakp)$ annihilates $M$.
\end{defn}

The main result of this subsection is the following proposition.
\begin{prop} \label{prop:critcl}
 Let $\alphaunder\in \QQ^{\Sigma_p}$ such that $\alpha_\frakp \leq \frac{w+k_{\tau_\frakp}-2}2$ for all $\frakp\in\Sigma_p$.
 Then there is a Hecke equivariant isomorphism
\[
S_{(\kunder,w)}(K_1(\frakN)^p\Iw_p,L_\sP)_\alphaunder
\isom
 \left( H^d_{\rig,!}(\Xbar_{K_1(\frakN)}^{\tor,\ord}, \mathscr{F}^{(\underline{k},w)}) \tensor_L L_\sP \right)_\alphaunder.
\]
\end{prop}
\begin{proof}
 If for all $\frakp\in\Sigma_p$, $\alpha_\frakp<\frac{w+k_{\tau_\frakp}-2}2$, then such an isomorphism is given by
\begin{align*}
 S_{(\kunder,w)}(K_1(\frakN)^p\Iw_p,L_\sP)_\alphaunder 
& \stackrel{\iota^\mathrm{cusp}}{\inj}  S^\dagger_{(\kunder,w)}(K_1(\frakN),L_\sP)_\alphaunder \\
&= \left( S^\dagger_{(\kunder,w)}(K_1(\frakN),L_\sP) / \Theta_\kunder \left(\Directsum_{\tau\in\Sigma_\infty}M^\dagger_{(s_{\Sigma_\infty\setminus\{\tau\}}\cdot \underline{k},w)}(K_1(\frakN),L_\sP) \right) \right)_\alphaunder \\
&\isom\left( H^d_{\rig,!}(\Xbar_{K_1(\frakN)}^{\tor,\ord}, \mathscr{F}^{(\underline{k},w)})\tensor_L L_\sP \right)_\alphaunder.
\end{align*}
Here $\iota^\mathrm{cusp}$ (see Lemma~\ref{lem:inj} for definition) induces an isomorphism on the slope $\alphaunder$ part by Theorem~\ref{thm:strongcl}.
The second equality is because if $\tau=\tau_\frakp$, then the image of $\Theta_{\tau,k_\tau-1}$ must have $U_\frakp$-slope at least $\frac{w+k_{\tau_\frakp}-2}2$ (\cite[Corollary~3.24]{TX16}).
The last isomorphism comes from combining the two isomorphisms in Theorem~\ref{thm:OCandRigid} for usual and compactly supported cohomology groups.

 Now let $\Sigma\subset\Sigma_p$ be a subset of primes $\frakp$ of $F$ above $p$.
 Let $\alphaunder$ be such that $\alpha_\frakp=\frac{w+k_{\tau_\frakp}-2}2$ for $\frakp\in \Sigma$ and $\alpha_\frakp<\frac{w+k_{\tau_\frakp}-2}2$ for $\frakp\notin \Sigma$.

Define a map $w\colon S_{(\kunder,w)}(K_1(\frakN)^p\Iw_p,L_\sP) \rightarrow S_{(\kunder,w)}(K_1(\frakN)^p\Iw_p,L_\sP)$ by 
\[
 w = \prod_{\frakp\in \Sigma} \left( (p-1)U_\frakp+\left(1-U_\frakp^2\cdot\left(\Nm_{F/\QQ}(\frakp)S_\frakp\right)^{-1}\right)w_\frakp \right),
\]
where $w_\frakp$ is the Atkin-Lehner operator at $\frakp$.
The map $w$ satisfies
\begin{align*} 
 w\circ S_\frakl &= S_\frakl \circ w  \\
 w\circ T_\frakl &= T_\frakl \circ w 
\end{align*}
where $\frakl$ is a prime of $F$ not dividing $\frakN p$.
Note that $w$ is defined so that it satisfies
\begin{align} \label{eq:wcomm}
 w \circ U_\frakp = 
\begin{cases}
\Fr_\frakp \circ w \text{\quad if $\frakp\in\Sigma$,} \\
U_\frakp \circ w \text{\quad if $\frakp\notin\Sigma$.} 		     
\end{cases}
\end{align}
Since $U_\frakp \Fr_\frakp = \Nm_{F/\QQ}(\frakp)S_\frakp$ (\cite[Lemma 3.20]{TX16}), $w$ restricts to an isomorphism
\[
 S_{(\kunder,w)}(K_1(\frakN)^p\Iw_p,L_\sP)_\alphaunder
 \xrightarrow{\sim}
 S_{(\kunder,w)}(K_1(\frakN)^p\Iw_p,L_\sP)_{\alphaunder'},
\]
where $\alphaunder'$ is such that $\alpha'_\frakp = w-1-\alpha_\frakp$ for $\frakp\in \Sigma$ and $\alpha'_\frakp = \alpha_\frakp$ for $\frakp\notin \Sigma$.
In particular, $\alpha'_\frakp < \frac{w-k_{\tau_\frakp}-2}2$ for all $\frakp\in \Sigma_p$.

We define a pairing between 
$\left( H^d_{\rig,!}(\Xbar_{K_1(\frakN)}^{\tor,\ord}, \mathscr{F}^{(\underline{k},w)})\tensor_L L_\sP \right)_\alphaunder$ and $S_{(\kunder,w)}(K_1(\frakN)^p\Iw_p,L_\sP)_\alphaunder$ by 
\[
 [x,y] \defeq \left(x, \varphi \circ w(y)  \right),
\]
where $\varphi\colon S_{(\kunder,w)}(K_1(\frakN)^p\Iw_p,L_\sP)_{\alphaunder'}
\xrightarrow{\sim}
\left( H^d_{\rig,!}(\Xbar_{K_1(\frakN)}^{\tor,\ord}, \mathscr{F}^{(\underline{k},w)}) \tensor_L L_\sP \right)_{\alphaunder'}$ 
is the isomorphism proven in the first paragraph.
Then the Hecke operators $T_\frakl$ and $S_\frakl$ are self-adjoint with respect to $[\;,\;]$ because they are self-adjoint with respect to $(\;,\;)$, and they commute with $w$.
Moreover, $U_\frakp$ is also self-adjoint with respect to $[\;,\;]$ since its adjoint operator with respect to $(\;,\;)$ is $\Fr_\frakp$ (Lemma~\ref{lem:Upadj}), and it commutes with $w$ in a twisted way as in Equation~(\ref{eq:wcomm}) above. 
The pairing $[\;,\;]$ is perfect because $(\;,\;)$ is perfect and $\varphi$ and $r$ are both isomorphisms.

From the pairing $[\;,\;]$, we obtain a $T_\frakl,S_\frakl,U_\frakp$-equivariant isomorphism between $S_{(\kunder,w)}(K_1(\frakN)^p\Iw_p,L_\sP)_\alphaunder$ and $\left( H^d_{\rig,!}(\Xbar_{K_1(\frakN)}^{\tor,\ord}, \mathscr{F}^{(\underline{k},w)}) \tensor_L L_\sP \right)_\alphaunder$.
In fact, one picks a Hecke eigenbasis $f_1,\ldots,f_m$ for $S_{(\kunder,w)}(K_1(\frakN)^p\Iw_p,L_\sP)$.
Then define a morphism
\[
 S_{(\kunder,w)}(K_1(\frakN)^p\Iw_p,L_\sP)_\alphaunder \rightarrow \left( H^d_{\rig,!}(\Xbar_{K_1(\frakN)}^{\tor,\ord}, \mathscr{F}^{(\underline{k},w)}) \tensor_L L_\sP \right)_\alphaunder
\]
by mapping $f_i$ to $x_i$ such that $[x_i,f_j]=\delta_{ij}$.
This is an isomorphism because $[\;,\;]$ is a perfect pairing.
Moreover, the fact that $T_\frakl,S_\frakl$, and $U_\frakp$ are self-adjoint with respect to $[\;,\;]$ implies that the morphism is Hecke equivariant.
\end{proof}
\begin{rmk}
 Proposition~\ref{prop:critcl} is a generalization of \cite[Lemma 7.3]{Col96}.
\end{rmk}

We now have the equivalence of (\ref{cond:geneigen}) and (\ref{cond:imtheta}) in Theorem~\ref{thm:main}.
\begin{cor} \label{cor:2and3}
 Let $f\in S_{(\kunder,w)}(K_1(\frakN)^p\Iw_p,L_\sP)$ be a Hecke eigenform of finite slope such that its $U_\frakp$-slope is not $\frac{w-1}2$ for any $\frakp\in\Sigma_p$.
 Then $f\in \Theta_\kunder\left(\Directsum_{\tau\in\Sigma_\infty}M^\dagger_{(s_{\Sigma_\infty\setminus\{\tau\}}\cdot \underline{k},w)}(K_1(\frakN), L_\sP)\right)$ if and only if there exists a generalized Hecke eigenform $f'\in S^\dagger_{(\underline{k},w)}(K_1(\frakN),L_\sP)$ with the same Hecke eigenvalues as $f$, but which is not a scalar multiple of $f$.
\end{cor}
\begin{proof}
 Suppose that $f\in \Theta_\kunder \left(\Directsum_{\tau\in\Sigma_\infty}M^\dagger_{(s_{\Sigma_\infty\setminus\{\tau\}}\cdot \underline{k},w)}(K_1(\frakN),L_\sP)\right)$.
 Since $f$ is classical, it follows from Proposition~\ref{prop:critcl} that the Hecke eigenvalues of $f$ appear in $H^d_{\rig,!}(\Xbar^{\tor,\ord},\mathscr{F}^{(\underline{k},w)})\tensor_{L} L_\sP$, which by Theorem~\ref{thm:OCandRigid} is isomorphic to
\[
 S^\dagger_{(\underline{k},w)}(K_1(\frakN),L_\sP)/\Theta_\kunder\left(\Directsum_{\tau\in\Sigma_\infty}M^\dagger_{(s_{\Sigma_\infty\setminus\{\tau\}}\cdot \underline{k},w)}(K_1(\frakN),L_\sP)\right).
\]
 Hence the generalized Hecke eigenspace of $S^\dagger_{(\underline{k},w)}(K_1(\frakN),L_\sP)$ containing $f$ cannot be entirely contained in the image of $\Theta_\kunder$.
 That is, there exists $f'\in S^\dagger_{(\underline{k},w)}(K_1(\frakN),L_\sP)$ not in $\Theta_\kunder\left(\Directsum_{\tau\in\Sigma_\infty}M^\dagger_{(s_{\Sigma_\infty\setminus\{\tau\}}\cdot \underline{k},w)}(K_1(\frakN),L_\sP)\right)$
having the same Hecke eigenvalues as $f$.
In particular, $f'$ is not a scalar multiple of $f$.

Conversely, suppose that there exists a generalized Hecke eigenform $f'\in S^\dagger_{(\underline{k},w)}(K_1(\frakN),L_\sP)$ with the same Hecke eigenvalues as $f$, but is not a scalar multiple of $f$.
On $S_{(\kunder,w)}(K_1(\frakN)^p\Iw_p,L_\sP)_{\alphaunder}$, given the prime-to-$p$ Hecke eigenvalues along with the $U_\frakp$-eigenvalue with $\frakp$-slope not equal to $\frac{w-1}2$ for any $\frakp\in\Sigma_p$, results of multiplicity one says that there is only one generalized Hecke eigenform up to scalar multiple.
Hence by Proposition~\ref{prop:critcl}, the same is true on $\left( H^d_{\rig,!}(\Xbar_{K_1(\frakN)}^{\tor,\ord}, \mathscr{F}^{(\underline{k},w)}) \tensor_L L_\sP \right)_\alphaunder$.
Then by Theorem~\ref{thm:OCandRigid}, $f$ and $f'$ span a $1$-dimensional subspace after quotient by the image of $\Theta_\kunder$.
Hence $f$ lies inside this image of $\Theta_\kunder$ (while $f'$ does not.)
\end{proof}

\section{Galois representations} \label{sec:Galrep}
Let us first recall how to how to associate a Galois representation to an overconvergent cuspidal Hilbert Hecke eigenform.
\begin{thm} \label{thm:Galrep}
Let $f\in S^\dagger_{(\kunder,w)}(K_1(\frakN), L_\sP)$ be an overconvergent cuspidal Hilbert Hecke eigenform.
For $\frakl$ a prime of $F$ not dividing $\frakN p$, let $\lambda_\frakl$ (resp. $\mu_\frakl$) be the $T_\frakl$ (resp. $S_\frakl$)-eigenvalue of $f$; 
for $\frakp\in\Sigma_p$, let $\lambda_\frakp$ be the $U_\frakp$-eigenvalue of $f$.
Then there exists a $p$-adic Galois representation 
\[
 \rho_f \colon \Gal_F \rightarrow \GL_2(\Qpbar)
\] 
satisfying the following.
\begin{enumerate}
 \item For every finite place $\frakl \nmid p\frakN$ of $F$, $\rho_f$ is unramified at $\frakl$ and
\[
 \det(T-\rho_f(\Frob_\frakl^{-1})) = T^2-\lambda_\frakl T+ \Nm_{F/\QQ}(\frakl)\mu_\frakl,
\]
where $\Frob_\frakl$ is the arithmetic Frobenius at $\frakl$.
 \item For every $\frakp\in \Sigma_p$, $\left.\rho_f\right|_{\Gal_{F_\frakp}}$ has Hodge--Tate--Sen weights $\frac{w-k_{\tau_\frakp}}2$, $\frac{w+k_{\tau_\frakp}-2}2$.
 \item $D_\cris(\left.\rho_f\right|_{\Gal_{F_\frakp}})^{\varphi=\lambda_\frakp}$ is non-zero and lies in 
$\Fil^{\frac{w-k_{\tau_\frakp}}2} D_{\cris}(\left.\rho_f\right|_{\Gal_{F_\frakp}})$.
\end{enumerate}
\end{thm}
\begin{proof}
 This is a theorem due to the work of many people.
 When $f$ is classical, the construction of $\rho_f$ and the verification of (1) was due to Carayol  when $d$ is odd and under an additional assumption when $d$ is even (\cite{Car86}).
 The method is to use Jacquet--Langlands correspondence to find the desired Galois representation in the cohomology of Shimura curves.
On the other hand, Wiles used $p$-adic variation of ordinary modular forms and the theory of pseudo-representations to deal with ordinary $f$ (\cite{Wiles88}).
Inspired by Wiles's method, Taylor completed the case when $d$ is even using congruences (\cite{Tay89}).
Blasius--Rogawski provided a different method which deals with odd and even $d$ at the same time (\cite{BR89}).
For (2), it is due to Saito's work on local-global compatibility at $p$ when $\rho_f$ comes from Carayol's construction (\cite{Saito09}), and due to Skinner for the remaining cases (\cite{Ski09}).

In general when $f$ is overconvergent, one uses the theory of pseudo-representations to construct Galois representations.

For (3), the existence of crystalline period is due to Kedlaya--Pottharst--Xiao (\cite{KPX14}) and Liu (\cite{Liu15}) independently, generalizing the work of Kisin for $F=\QQ$ (\cite{Kis03}).
\end{proof}

\begin{prop} \label{prop:3to4}
 Let $f\in S_{(\kunder,w)}(K_1(\frakN)\Iw_p, L_\sP)$ be a classical cuspidal Hilbert Hecke eigenform of finite slope.
For $\frakl$ a prime of $F$ not dividing $\frakN p$, let $\lambda_\frakl$ be the $T_\frakl$-eigenvalue of $f$; 
for $\frakp\in\Sigma_p$, let $\lambda_\frakp$ be the $U_\frakp$-eigenvalue of $f$.
 If $f\in \Theta_\kunder\left(\Directsum_{\tau\in\Sigma_\infty}M^\dagger_{(s_{\Sigma_\infty\setminus\{\tau\}}\cdot \underline{k},w)}(K_1(\frakN), L_\sP)\right)$,
then for some $\frakp\in\Sigma_p$, $\left.\rho_f\right|_{\Gal_{F_\frakp}}$ splits and $\val_p(\lambda_\frakp) = \frac{w+k_{\tau_\frakp}-2}2$.
\end{prop}
\begin{proof}
 We first show that there exists $\frakp\in\Sigma_p$ and $g\in M^\dagger_{(s_{\Sigma_\infty\setminus\{\tau_\frakp\}}\cdot \underline{k},w)}(K_1(\frakN), L_\sP)$ such that $f=\Theta_{\tau_\frakp,k_{\tau_\frakp}-1}(g)$.
This would in particular implies $\val_p(\lambda_\frakp) = \frac{w+k_{\tau_\frakp}-2}2$.
 In fact, since $f$ is classical, for each $\frakp\in\Sigma_p$ its $\frakp$-slope $\alpha_\frakp \defeq \val_p(\lambda_\frakp)$ satisfies $\frac{w-k_{\tau_\frakp}}2 \leq \alpha_\frakp \leq \frac{w+k_{\tau_\frakp}-2}2$.
Note that $\Theta_{\tau,k_\tau-1}$ is Hecke equivariant, and hence the $\frakp$-slope of $g$ is also $\alpha_\frakp$.
However, the $\frakp$-slope of $M^\dagger_{(s_{\Sigma_\infty\setminus\{\tau_\frakp\}}\cdot \underline{k},w)}(K_1(\frakN), L_\sP)$ must be at least $\frac{w+k_{\tau_\frakp}-2}2$ (\cite[Corollary~3.24]{TX16} or Theorem~\ref{thm:Galrep}(3)).
Hence $\alpha_\frakp=\frac{w+k_{\tau_\frakp}-2}2$.

By assumption,
\[
 f\in \Theta_\kunder\left(\Directsum_{\frakp\in\Sigma_p}M^\dagger_{(s_{\Sigma_\infty\setminus\{\tau_\frakp\}}\cdot \underline{k},w)}(K_1(\frakN), L_\sP)_{\alphaunder}\right),
\]
Since each $M^\dagger_{(s_{\Sigma_\infty\setminus\{\tau_\frakp\}}\cdot \underline{k},w)}(K_1(\frakN), L_\sP)_{\alphaunder}$ is finite dimensional, we choose a basis $g_{\frakp,i}, i=1,\ldots, r_\frakp$ of it consisting of generalized Hecke eigenforms.
Write $f=\sum_{\frakp\in\Sigma_p} \sum_{i=1}^{r_\frakp} a_{\frakp,i} \Theta_{\tau_\frakp,k_{\tau_\frakp}-1}(g_{\frakp,i})$.
 Since $\Theta_{\tau,k_\tau-1}$ is Hecke equivariant, we know that
 $\Theta_{\tau_\frakp,k_{\tau_\frakp}-1}(g_{\frakp,i})$ is still a generalized Hecke eigenform.
 Since $f$ is Hecke eigen, all the $\Theta_{\tau_\frakp,k_{\tau_\frakp}-1}(g_{\frakp,i})$'s with $a_{\frakp,i}\neq 0$ must have the same Hecke eigenvalues as $f$ and are Hecke eigenforms.
 Results of multiplicity one implies that they are scalar multiples of each other, and in particular scalar multiples of $f$.
 Choosing $g$ to be a suitable scalar multiple of a $g_{\frakp,i}$ with $a_{\frakp,i}\neq0$, our claim that $f=\Theta_{\tau_\frakp,k_{\tau_\frakp}-1}(g)$ is proved.

 Once we know $f = \Theta_{\tau_\frakp,k_{\tau_\frakp}-1}(g)$, a similar argument to \cite[Theorem 6.6 (2)]{Kis03} shows that $\left.\rho_f\right|_{\Gal_{F_\frakp}}$ must be split.
\end{proof}

\section{Eigenvariety} \label{sec:eigenvar}
We recall the cuspidal Hilbert eigenvariety constructed by Andreatta--Iovita--Pilloni (\cite{AIP16}).

Recall that $G$ is the algebraic group $\Res_{\cO_F/\ZZ}\GL_{2}$.
We first define the \emph{weight space} $\cW$ for $G$.
We saw in Section~\ref{subsec:weights} that the weights for $G$ are tuples $(\kunder,w)\in\ZZ^{\Sigma_\infty}\times \ZZ$ such that $k_\tau \congr w \pmod 2$.
The subgroup of these tuples in $\ZZ^{\Sigma_\infty}\times \ZZ$ is isomorphic to $\ZZ^{\Sigma_\infty}\times \ZZ$ via $((k_\tau)_\tau,w) \mapsto ((\nu_\tau)_\tau,w)\defeq ((\frac{w-k_\tau}2)_\tau,w)$.
Since $\ZZ^{\Sigma_\infty} \times \ZZ$ is the character group of $\Res_{\cO_F/\ZZ}\Gm \times \Gm$, we define the weight space of $G$ to be 
\[
 \cW \defeq \Spf(\ZZ_p[\![(\Res_{\cO_F/\ZZ}\Gm\times\Gm)(\ZZ_p)]\!])^\rig.
\]

Let $\cU = \Sp A_\cU$ be an affinoid with a morphism of ringed spaces $\kappa^\cU \colon \cU \rightarrow \cW$.
Andreatta--Iovita--Pilloni constructed a Fr\'echet $A_\cU$-module $M^\dagger(K_1(\frakN),\kappa^\cU)$ (resp. $S^\dagger(K_1(\frakN),\kappa^\cU)$), called \emph{the module of $p$-adic families of overconvergent (resp. cuspidal) Hilbert modular forms with weights parametrized by $\cU$} (\cite[Definition 4.2]{AIP16}). 
In particular, for $\cU = \Sp \CC_p$ and $[\kappa\colon\cU \rightarrow \cW] \in \cW(\CC_p)$, this construction gives a $\CC_p$-vector space $M^\dagger(K_1(\frakN),\kappa)$ (resp. $S^\dagger(K_1(\frakN),\kappa)$), called \emph{the space of overconvergent (resp. cuspidal) Hilbert modular forms of weight $\kappa$ and level $K_1(\frakN)$ with coefficients in $\CC_p$}.

\begin{lem}
 Let $\kappa\in \cW(\CC_p)$.
Assume that $\kappa$ is a classical weight in the sense that it corresponds to $((\frac{w-k_\tau}2)_\tau,w) \in \ZZ^{\Sigma_\infty}\times\ZZ$.
 Then 
\begin{align*}
 M^\dagger(K_1(\frakN),\kappa) &= M^\dagger_{(\kunder,w)}(K_1(\frakN),\CC_p) \text{ and}  \\ 
 S^\dagger(K_1(\frakN),\kappa) &= S^\dagger_{(\kunder,w)}(K_1(\frakN),\CC_p),
\end{align*}
where the right hand sides were defined in Section~\ref{subsec:OC}.
\end{lem}
Hence to align with previous notations, we also use the notation $M^\dagger_\kappa(K_1(\frakN),\CC_p)$ (resp. $S^\dagger_\kappa(K_1(\frakN), \CC_p$) for $M^\dagger(K_1(\frakN),\kappa)$ (resp. $S^\dagger(K_1(\frakN),\kappa)$).
\begin{proof}
 By definition, $M^\dagger_{(\kunder,w)}(K_1(\frakN),\CC_p) = M^\dagger_{(\kunder,w)}(K',\CC_p)^{K_1(\frakN)/K'}$, where $K'\subset K_1(\frakN)$ is chosen such that $K'$ satisfies (\ref{cond:level}).
 As in Section~\ref{sec:geom}, $X_{K'}$ is a disjoint union of $\left(M_{K'}^\frakc/\Delta_{K'}\right) \tensor_{\ZZ_{(p)}} W(k_0)$, where $\Delta_{K'} = \cO_F^{\times,+}/(K'\cap \cO_F^\times)^2$ and $\frakc$ is a fractional ideal of $F$ and runs through a fixed set of representatives for $\operatorname{Cl}^+(F)$.
 Choose integral toroidal compactifications $M_{K'}^{\frakc,\tor}$ of $M_{K'}^\frakc$ compatible with $X_{K'}^\tor$ in the sense that $X_{K'}^\tor$ is a disjoint union of $\left(M_{K'}^{\frakc,\tor}/\Delta_{K'}\right)\tensor_{\ZZ_{(p)}} W(k_0)$.
To simplify notation, let $Y_{K'}^\frakc$, $Y_{K'}^{\frakc,\tor}$ be $M_{K'}^\frakc$, $M_{K'}^{\frakc,\tor}$ based changed from $\ZZ_{(p)}$ to $W(k_0)$, and $\cY_{K'}^{\frakc,\tor}$ the rigid generic fiber of the formal completion of $Y_{K'}^{\frakc,\tor}$ along its special fiber, based changed from $W(k_0)$ to $L_\sP$.
 Then 
\begin{align*}
M^\dagger_{(\kunder,w)}(K',\CC_p)^{K_1(\frakN)/K'}
&=H^0(\cX_{K'}^\tor,j^\dagger\omegaunder^{(\kunder,w)})^{K_1(\frakN)/K'} \\
&=\Directsum_\frakc H^0(\cY_{K'}^{\frakc,\tor}/\Delta_{K'}, j^\dagger\omegaunder^{(\kunder,w)})^{K_1(\frakN)/K'} \\
&=\Directsum_\frakc H^0(\cY_{K'}^{\frakc,\tor}, j^\dagger\omegaunder^{(\kunder,w)})^{\Delta_{K'},K_1(\frakN)/K'} \\
&=\Directsum_\frakc H^0(\cY_{K_1(\frakN)}^{\frakc,\tor}, j^\dagger\omegaunder^{(\kunder,w)})^{\Delta_{K'}} \\
&=\Directsum_\frakc H^0(\cY_{K_1(\frakN)}^{\frakc,\tor}, j^\dagger \omegaunder^{(\kunder,w)})^{\Delta_{K}}.
\end{align*}
Here the third equality used the fact that from the choice of $K'$, the quotient by $\Delta_{K'}$ gives isomorphism of geometric components of $M_{K'}^\frakc$ onto its image.
 When $\kappa$ is a classical weight, the modular sheaf of weight $\kappa$ is the classical modular sheaf (\cite[Corollary 3.10]{AIP16}).
 Hence the last term is the definition of $M^\dagger(K_1(\frakN),\kappa)$ (\cite[Definition 4.1, 4.6]{AIP16}).

 The cuspidal case follows by the same argument.
\end{proof}

There is a corresponding quasi-coherent sheaf of overconvergent (resp. cuspidal) Hilbert modular forms $M^\dagger(K_1(\frakN))$ (resp. $S^\dagger(K_1(\frakN))$) over $\cW$, whose value on an admissible affinoid open $\cU \subset \cW$ is $M^\dagger(K_1(\frakN),\cU)$ (resp. $S^\dagger(K_1(\frakN),\cU)$).
Moreover, any overconvergent cuspidal Hilbert modular form can be put in a $p$-adic family:
\begin{prop}[{\cite[Theorem 4.4]{AIP16}}]
 Let $\cU\subset \cW$ be an admissible affinoid open, and $\kappa\in \cU(\CC_p)$.
 Then the specialization map 
\[
 S^\dagger(K_1(\frakN))(\cU) \rightarrow S^\dagger_\kappa(K_1(\frakN), \CC_p)
\]
 is surjective.
\end{prop}

Let $\cH^{\frakN p}$ be the abstract Hecke algebra away from $\frakN p$.
This is a commutative $\QQ_p$-algebra generated by the operators $T_\frakl$ and $S_\frakl$ for $\frakl$ prime to $\frakN p$.
Let $\cU_p$ be the $\QQ_p$-algebra generated by the $U_\frakp$-operators for all $\frakp\in\Sigma_p$.
Andreatta--Iovita--Pilloni defined an action of the algebra $\cH^{\frakN p}\tensor_{\QQ_p} \cU_p$ on $S^\dagger(K_1(\frakN))$ (\cite[Sec.\ 4.3]{AIP16}).
Moreover, $U_p$ is a compact operator (\cite[Lemma 3.27]{AIP16}).
Then by Buzzard's eigenvariety machine (\cite[Construction 5.7]{Buz07}), there exists a rigid analytic space $\cE_\frakN$, the eigenvariety associated to $(\cW,S^\dagger(K_1(\frakN)),\cH^{\frakN p}\tensor_{\QQ_p}\cU_p,, U_p)$, as well as a weight map $\wt\colon \cE_\frakN \rightarrow \cW$.
We sketch the construction: for any admissible affinoid open $\cU \subset \cW$, let $\cZ_\cU$ be the spectral variety of $U_p$, i.e., the closed subspace of $\cU \times \AA^1$ cut out by the characteristic series of $U_p$.
There is an admissible cover of $\cZ_\cU$, consisting of affinoid subdomains $\cV$ of $\cZ_\cU$ such that there exists an affinoid subdomain $\cU'$ of $\cU$ with the preimage of $\cU'$ containing $\cV$ and also $\cV$ surjecting onto $\cU'$ (\cite[Theorem 4.6]{Buz07}).
Since over $\cV$ only finitely many non-zero $U_p$-eigenvalues can show up, one can split off the finite-dimensional $U_p$-generalized eigenspace $N$ of $S^\dagger(K_1(\frakN))(\cU)$ corresponding to these finitely many $U_p$-eigenvalues.
Let $\cH(\cV)$ be the image of $\cH^{\frakN p}\tensor_{\QQ_p} \cU_p$ inside $\End_{\cO(\cU')}N$.
Then $\operatorname{Sp}\cH(\cV) \rightarrow \cV$ is a finite morphism, and the $\operatorname{Sp}\cH(\cV)$'s glue into $\cE_{\frakN}$.

\begin{center} 
\begin{tikzcd}
 \cV \ar[r,symbol=\subset] \ar[rd,dotted,twoheadrightarrow] & \cZ_\cU \times_{\cU} \cU' \ar[r,symbol=\subset] \ar[d, twoheadrightarrow] & \cZ_\cU \ar[d,twoheadrightarrow] \\
 & {}^\exists \cU' \ar[r,symbol=\subset] & \cU 
\end{tikzcd}
\end{center}

\begin{rmk}
 In \cite{AIP16}, Andreatta--Iovita--Pilloni actually used the $\mu_N$-level where $N\in\ZZ_{\geq4}$ instead of $K_1(\frakN)$-level. 
 But the construction works through the more general setting.
\end{rmk}

We summarize some important properties of the cuspidal Hilbert eigenvariety.
\begin{thm}[{\cite[Theorem 5.1]{AIP16}}] \hfill
 \begin{enumerate} 
  \item The cuspidal Hilbert eigenvariety $\cE_\frakN$ is equidimensional of dimension $d+1$.
  \item The weight map $\wt$ is, locally on $\cE_\frakN$ and $\cW$, finite and surjective.
  \item For all $\kappa\in \cW(\CC_p)$, $\wt^{-1}(\kappa)$ is in bijection with the finite-slope Hecke eigenvalues appearing in $S^\dagger_{\kappa}(K_1(\frakN),\CC_p)$.
  \item There is a universal Hecke character $\lambda\colon\cH^{\frakN p}\tensor_{\QQ_p} \cU_p \rightarrow \cO_{\cE_\frakN}$, glued from $\cH^{\frakN p} \tensor_{\QQ_p} \cU_p \surj \cH(\cV)$.
  \item There is a universal pseudo-character
	\[
	\cT \colon \Gal_{F}\rightarrow\cO_{\cE_\frakN}
	\]
	unramified outside $p\frakN$ such that $\cT(\Frob_\frakl^{-1})=\lambda(T_\frakl)$ for all prime $\frakl$ of $\cO_F$ prime to $\frakN p$.
	Here as before $\Frob_\frakl$ is the arithmetic Frobenius at $\frakl$.
  \item For all $x\in \cE_\frakN$, there is a semisimple Galois representation
	\[
	\rho_x \colon \Gal_{F}\rightarrow\GL_2(\bar{k}(x)) 
	\]
	characterized by $\Tr(\rho_x) = \left.\cT\right|_x$, $\det(\rho_x)(\Frob_\frakl^{-1}) = \Nm_{F/\QQ}(\frakl) \left.\lambda\right|_x (S_\frakl)$.
Here $\left.\right|_x$ denotes composing with the specialization map $\cO_{\cE_\frakN} \rightarrow \bar{k}(x)$ to the residue field $\bar{k}(x)$ of $x$.
 \end{enumerate}
\end{thm}

Below we will prove the equivalence of (\ref{cond:ram}) and (\ref{cond:geneigen}) in Theorem~\ref{thm:main}.
Recall that if $f\colon \cX \rightarrow \cY$ is a morphism of rigid analytic varieties, then $f$ is \emph{\'etale at $x\in \cX$} if $\cO_{\cX,x}$ is flat over $\cO_{\cY,f(x)}$ and $\cO_{\cX,x}/\frakm_{f(x)}\cO_{\cX,x}$ is a finite separable field extension of the residue field $\cO_{\cY,f(x)}/\frakm_{f(x)}$ of $f(x)$ (\cite[Definition 1.7.10]{Huber}).
\begin{lem} \label{lem:1and2}
 Let $f\in S^\dagger_{\kappa}(K_1(\frakN), \CC_p)$ be an overconvergent cuspidal Hilbert Hecke eigenform of finite slope.
 Let $x \in \cE_\frakN$ be the point corresponding to $f$.
 Then $x$ is a non-\'etale point with respect to $\wt \colon \cE_\frakN \rightarrow \cW$ if and only if there exists an overconvergent cuspidal Hilbert generalized Hecke eigenform $f'$ with the same Hecke eigenvalues and weight as $f$, but which is not a scalar multiple of $f$.
\end{lem}
\begin{proof}
 We have $\kappa = \wt(x)$.
 Since $\wt$ is locally-on-the-domain finite flat, $\cO_{\cE_\frakN,x}$ is flat over $\cO_{\cW,\kappa}$.
 Also $\cO_{\cW,\kappa}/\frakm_{\kappa}$ is a field of characteristic $0$, for which all finite field extensions are separable.
 Hence by defintion, $x$ is a non-\'etale point with respect to $\wt$ if and only if $\cO_{\cE_{\frakN},x}/\frakm_{\kappa}\cO_{\cE_{\frakN},x}$ is not a field.
This means that $\frakm_{\kappa} \cO_{\cE_\frakN,x} \subsetneq \frakm_x$.
Or equivalently there exists an overconvergent cuspidal Hilbert modular form $f'\in S^\dagger_{\kappa}(K_1(\frakN), L_\sP)$ annihilated by some power of $\frakm_x$ but not $\frakm_x$ itself, 
i.e., $f'$ is a generalized Hecke eigenform with the same Hecke eigenvalues as $f$, but which is not a scalar multiple of $f$.
\end{proof}

\section{Galois deformations} \label{sec:Galdeform}
Let $(\kunder,w)$ be a cohomological weight.
Let $f\in S_{(\kunder,w)}(K_1(\frakN)\Iw_p, L_\sP)$ be a Hecke eigenform of finite slope.
 For each prime $\frakp\in\Sigma_p$, let $\lambda_\frakp$ be the $U_\frakp$-eigenvalue of $f$.
 Assume that $\val_p(\lambda_\frakp)\neq\frac{w-1}2$ for any $\frakp\in\Sigma_p$. 
 Let $x\in\cE$ be the point on the cuspidal Hilbert eigenvariety $\cE \defeq \cE_\frakN$ corresponding to $f$.
 Let $\rho \colon \Gal_F \rightarrow \GL(V)$ be the $p$-adic Galois representation corresponding to $f$ as in Theorem~\ref{thm:Galrep}.
Here $V$ is a $2$-dimensional $\bar{k}(x)$-vector space, 
where $\bar{k}(x)$ is the residue field of $x\in\cE$.
For all $\frakp\in\Sigma_p$, we have $\Fil^{\frac{w-k_{\tau_\frakp}}2}D_\cris(\left.V\right|_{\Gal_{F_\frakp}})^{\varphi=\lambda_\frakp}\neq 0$.

The goal of this section is to use Galois deformation theory to prove the following theorem.
\begin{thm} \label{thm:4to1}
 If there exists $\frakp\in \Sigma_p$ such that $\left.\rho\right|_{\Gal_{F_\frakp}}$ splits and $\val_p(\lambda_\frakp) = \frac{w+k_{\tau_\frakp}-2}2$, then $x$ is a ramification point of $\cE$.
 Moreover, the tangent space of the fiber of $\wt$ at $x$ has dimension $\geq\#\{\frakp\in\Sigma_p \mid \text{$\left.\rho\right|_{\Gal_{F_\frakp}}$ splits and $\val_p(\lambda_\frakp) = \frac{w+k_{\tau_\frakp}-2}2$}\}$.
\end{thm}

\subsection{Galois deformation rings} \label{subsec:Galdeformring}
In this subsection, we define various Galois deformation rings needed in the proof of Theorem~\ref{thm:4to1}.

We define a deformation functor $D$ on the category of Artinian local $\bar{k}(x)$-algebras with residue field $\bar{k}(x)$.
For any such $\bar{k}(x)$-algebra $A$, let $D(A)$ be the set of strict equivalence classes of continuous representations $\rho_A \colon \Gal_F \rightarrow \GL(V_A)$ deforming $\rho$ such that
\begin{enumerate}
 \item \label{cond:min}
for all primes $\frakl$ of $F$ not dividing $p$, $\rho_A$ and $\rho$ are the same after restricting to the inertia subgroup at $\frakl$,
 \item \label{cond:HT}
for $\frakp\in\Sigma_p$, the sum of the two Hodge--Tate--Sen weights of $\left.V_A\right|_{\Gal_{F_\frakp}}$ is independent of $\frakp$, and
 \item \label{cond:periodmap}
for all $\frakp\in\Sigma_p$, there exists a lift $\widetilde{\lambda}_\frakp\in A$ of $\lambda_\frakp$ such that $\Fil^{\frac{\widetilde{w}-\widetilde{k}_{\tau_\frakp}}2}D_\cris(V_A|_{\Gal_{F_\frakp}})^{\varphi=\widetilde\lambda_\frakp}\neq0$, where $\widetilde{w}, \widetilde{k}_{\tau_\frakp}\in A$ are lifts of $w,k_{\tau_\frakp}$ such that $\frac{\widetilde{w}-\widetilde{k}_{\tau_\frakp}}2$ and $\frac{\widetilde{w}+\widetilde{k}_{\tau_\frakp}-2}2$ are Hodge--Tate--Sen weights of $V_A|_{\Gal_{F_\frakp}}$.
\end{enumerate}

Let $D^0$ be the sub-functor of $D$ of deformations with the additional condition
\begin{enumerate}[resume]
 \item \label{cond:fiber}
$\rho_A$ has constant $\frakp$-Hodge--Tate--Sen weights for all $\frakp\in\Sigma_p$.
\end{enumerate}

\begin{lem} \label{lem:rep}
 $D$ and $D^0$ are pro-representable by some complete local $\bar{k}(x)$-algebras $R$ and $R^0$, respectively.
\end{lem}
\begin{proof}
 Since $f$ is classical and cuspidal, $\rho_f$ is absolutely irreducible. 
 Hence the full deformation functor is pro-representable (\cite[\textsection 10]{Ma} \cite[Lemma 9.3]{Kis03}).
 Condition~(\ref{cond:min}) is a deformation condition by the proof of \cite[Proposition~7.6.3(i)]{BC}.
 Since we assume that for any $\frakp\in\Sigma_p$, the $U_\frakp$-slope of $f$ is not $\frac{w-1}2$, the $\varphi$-eigenvalues on $D_\cris(\left.V\right|_{\Gal_{F_\frakp}})$ has multiplicity one.
 Hence we may apply \cite[Proposition~8.13]{Kis03}, which says that condition~(\ref{cond:periodmap}) is a deformation condition.

Let $S$ be the universal deformation ring pro-representing the deformation functor of $\rho$ with condition~(\ref{cond:min}) and (\ref{cond:periodmap}).
For $\frakp\in \Sigma_p$, let $S_\frakp$ be the versal deformation ring for $\left.\rho\right|_{\Gal_{F_\frakp}}$.
Write $\varphi_\frakp \colon \Spf S \rightarrow \Spf S_\frakp$ for the map induced by restricting a Galois deformation to the decomposition group at $\frakp$.
By \cite[Theorem, p.659]{Sen88}, given $\frakp\in\Sigma_p$, the sum of Hodge--Tate--Sen weights of the universal deformation $\left.V_S\right|_{\Gal_{F_\frakp}}$ is an analytic function $f_\frakp$ on $\Spf S_p$.
One can then construct the universal deformation ring $R$ of $\rho$ with condition~(\ref{cond:min}), (\ref{cond:HT}) and (\ref{cond:periodmap}) by taking the quotient of $S$ by the ideal generated by $\varphi_\frakp^\ast f_\frakp - \varphi_{\frakp'}^\ast f_{\frakp'}$ with $\frakp, \frakp'\in\Sigma_p$ distinct $p$-adic primes of $F$.
 We thus conclude the pro-representability of $D$.

As for $D^0$, since we assumed the weight of $f$ is cohomological, for any $\frakp\in\Sigma_p$, the two $\frakp$-Hodge--Tate--Sen weights are distinct.
Hence the condition that a deformation has a constant $\frakp$-Hodge--Tate--Sen weights can be described as the vanishing of symmetric polynomials in the two $\frakp$-Hodge--Tate--Sen weights.
Again using \cite[Theorem, p.659]{Sen88}, one may construct the universal deformation ring $R^0$ of $D^0$ as a quotient of $R$.
\end{proof}

\begin{prop} \label{prop:R=T} \hfill
\begin{enumerate} 
 \item The tangent space $T_x\cE\defeq \Hom_{\bar{k}(x)}(\cO_{\cE,x},\bar{k}(x)[\varepsilon]/\varepsilon^2)$ of $\cE$ at $x$ is a subspace of $D(\bar{k}(x)[\varepsilon]/\varepsilon^2)$.
 \item Let $\cE_{\wt(x)}\defeq \cE \times_\cW \wt(x)$ be the fiber of $\wt$ at $\wt(x)$.
       Then the tangent space $T_x\cE_{\wt(x)}$ of the fiber is the intersection of $T_x\cE$ and $D^0(\bar{k}(x)[\varepsilon]/\varepsilon^2)$.
\end{enumerate}
\end{prop}
\begin{proof} \hfill
\begin{enumerate} 
 \item  
 By the assumption that $f$ is classical and cuspidal, $\rho \colon \Gal_F \rightarrow \GL(V) \isom \GL_2(\bar{k}(x))$ is absolutely irreducible.
Then since $\cO_{\cE,x}$ is Henselian, the theorem of Nyssen and Rouquier (\cite[Th\'eor\`eme 1]{Nys}\cite[Corollaire 5.2]{Rou}) implies that there exists a Galois representation
 \[
 \Gal_F \rightarrow \GL_2(\cO_{\cE,x})
 \]
 whose residual representation is $\rho$ and whose trace gives the pseudo-character $\cT$ composed with $\cO_\cE \rightarrow \cO_{\cE,x}$.
 Since the Galois representation satisfies the conditions in the deformation functor $D$, we have a morphism $R \rightarrow \cO_{\cE,x}$.
 Note that by construction of $\cE$, $\cO_{\cE,x}$ is generated by the prime-to-$\frakN p$ Hecke eigenvalues and the $U_\frakp$-eigenvalues.
 They are traces of Frobenius and the crystalline period $\widetilde{\lambda}_\frakp$ of the universal Galois representation, respectively, and hence they lie in the image of the universal deformation ring $R$.
We thus conclude that the morphism $R \rightarrow \cO_{\cE,x}$ is surjective.
 This induces an injection on the Zariski tangent spaces $T_x\cE \inj \Hom(R,\bar{k}(x)[\varepsilon]/\varepsilon^2) = D(\bar{k}(x)[\varepsilon]/\varepsilon^2)$.
 \item From Theorem~\ref{thm:Galrep}(2), we know that the Hodge--Tate--Sen weights and the weight are determined by each other.
\end{enumerate}
\end{proof}

We will also need some auxiliary deformation sub-functors of $D$.
Given $\frakp\in\Sigma_p$, let $D^\frakp\subset D$ be the sub-functor of deformations with a constant $\frakp$-Hodge--Tate--Sen weight $\frac{w-k_{\tau_\frakp}}2$.

\begin{lem} \label{lem:R^p}
 $D^\frakp$ is pro-representable by some complete local $\bar{k}(x)$-algebra $R^\frakp$ and $\dim R^\frakp \geq \dim R-1$.
\end{lem}
\begin{proof}
 As in the proof of Lemma~\ref{lem:rep}, let $\varphi_{\frakp}\colon \Spf R \rightarrow \Spf S_{\frakp}$ be the map induced by restricting a Galois deformation to the decomposition group at $\frakp$.
Since we assumed the weight of $f$ is cohomological, the two $\frakp$-Hodge--Tate--Sen weights are distinct.
 Hence the condition that a deformation has a constant $\frakp$-Hodge--Tate--Sen weight $\frac{w-k_{\tau_0}}2$ can be described as the vanishing of a symmetric polynomial $\Phi$ in the two $\frakp$-Hodge--Tate--Sen weights.
 Again by \cite[Theorem, p.659]{Sen88}, this symmetric polynomial $\Phi$ is an analytic function on $\Spf S_\frakp$.
 Hence the universal deformation ring $R^\frakp$ can be consructed as the quotient $R/(\varphi_{\frakp}^\ast\Phi)$.
 The claim about the Krull dimension then follows.
\end{proof}

\subsection{Computing tangent spaces}
In this subsection, we will compute the codimension of $D^0(\bar{k}(x)[\varepsilon]/\varepsilon^2)$ in $D(\bar{k}(x)[\varepsilon]/\varepsilon^2)$ and deduce Theorem~\ref{thm:4to1}.

We begin to use the assumptions of $f$ in Theorem~\ref{thm:4to1} that there exists $\frakp\in\Sigma_p$ such that $\left.\rho\right|_{\Gal_{F_\frakp}}$ splits and $\val_p(\lambda_\frakp)=\frac{w+k_{\tau_\frakp}-2}2$.
Let $\Sigma\subset\Sigma_p$ be a subset such that $\left.\rho\right|_{\Gal_{F_\frakp}}$ splits and $\val_p(\lambda_\frakp)=\frac{w+k_{\tau_\frakp}-2}2$ for all $\frakp\in \Sigma$.
Write $\left.\rho\right|_{\Gal_{F_\frakp}} = \psi_{\frakp,1}\directsum\psi_{\frakp,2}$ for $\frakp\in \Sigma$.
Without loss of generality, we assume that for $\frakp\in\Sigma$, $\val_p(\lambda_\frakp)$ is the Hodge--Tate weight of $\psi_{\frakp,2}$.

It is known that first order deformations of a representation is equivalent to self-extensions of the representation.
Hence $D(\bar{k}(x)[\varepsilon]/\varepsilon^2)$ is a subspace of $\Ext^1_{\Gal_F}(V,V)$, which is further identified with $H^1(\Gal_F,V\tensor V^\ast)$.
For any $\frakp\in\Sigma_p$, we write $\loc_\frakp$ for the restriction map 
\[
 \loc_\frakp \colon H^1(\Gal_F,V\tensor V^\ast) \rightarrow H^1(\Gal_{F_\frakp},V\tensor V^\ast).
\]

Let $\widetilde{V}\in D(\bar{k}(x)[\varepsilon]/\varepsilon^2)$.
Then for all $\frakp\in\Sigma$, 
$\loc_\frakp(\widetilde{V})\in H^1(\Gal_{F_\frakp}, V\tensor V^\ast)=\Directsum_{i,j=1}^2 H^1(\Gal_{F_\frakp},\psi_{\frakp,i}\psi_{\frakp,j}^{-1})$.
Write $\loc_\frakp(\widetilde{V}) = (e_{\frakp,ij})$ according to this decomposition.

\begin{lem} \label{lem:22crys}
Let $\frakp\in\Sigma$, so that $e_{\frakp,22}$ makes sense.
Then $e_{\frakp,22}$ is a crystalline cohomology class, i.e. $e_{\frakp,22}$ lies in the kernel of 
\[
 H^1(\Gal_{F_\frakp}, \psi_{\frakp,2}\psi_{\frakp,2}^{-1}) \rightarrow H^1(\Gal_{F_\frakp}, \psi_{\frakp,2}\psi_{\frakp,2}^{-1}\tensor B_\cris).
\]
\end{lem}
\begin{proof}
From the short exact sequence
\[
 0 \rightarrow V \rightarrow \widetilde{V} \rightarrow V \rightarrow 0,
\]
by taking $D_\cris(\cdot) = (\cdot\tensor B_\cris)^{\Gal_{F_{\frakp}}}$ we get a long exact sequence 
\[
 0 \rightarrow D_\cris(\left.V\right|_{\Gal_{F_\frakp}}) \rightarrow D_\cris(\left.\widetilde{V}\right|_{\Gal_{F_\frakp}}) \rightarrow D_\cris(\left.V\right|_{\Gal_{F_\frakp}}) \rightarrow H^1(\Gal_{F_\frakp},V\tensor B_\cris).
\]
 Condition~(\ref{cond:periodmap}) in the deformation functor $D$ says that there exists a lift $\widetilde{\lambda}_\frakp\in \bar{k}(x)[\varepsilon]/\varepsilon^2$ of $\lambda_\frakp$ such that
\[
 D_\cris(\left.\widetilde{V}\right|_{\Gal_{F_\frakp}})^{\varphi=\widetilde{\lambda}_\frakp}\neq 0.
\]
Since $f$ has cohomological weight, $D_\cris(V)^{\varphi=\lambda_\frakp}$ is $1$-dimensional (but not larger) over $\bar{k}(x)$.
Hence we have the short exact sequence
\[
  0 \rightarrow D_\cris(\left.V\right|_{\Gal_{F_\frakp}})^{\varphi=\lambda_\frakp} \rightarrow D_\cris(\left.\widetilde{V}\right|_{\Gal_{F_\frakp}})^{\varphi=\widetilde\lambda_\frakp} \rightarrow D_\cris(\left.V\right|_{\Gal_{F_\frakp}})^{\varphi=\lambda_\frakp} \rightarrow 0.
\]
Since $\val_p(\lambda_\frakp) = \frac{w+k_{\tau_\frakp}-2}2$, $D_\cris(\left.V\right|_{\Gal_{F_\frakp}})^{\varphi=\lambda_\frakp} = D_\cris(\psi_{\frakp,2})$.
Hence the above short exact sequence means
\[
 0\rightarrow D_\cris(\psi_{\frakp,2}) \rightarrow D_\cris(e_{\frakp,22}\psi_{\frakp,2}) \rightarrow D_\cris(\psi_{\frakp,2}) \rightarrow 0.
\]
Here $e_{\frakp,ij}\psi_{\frakp,j}$ stands for the extension of $\psi_{\frakp,j}$ by $\psi_{\frakp,i}$ corresponding to the cohomology class $e_{\frakp,ij}$.
The surjectivity means that the cohomology class $e_{\frakp,22}\in H^1(\Gal_{F_\frakp},\psi_{\frakp,2}\psi_{\frakp,2}^{-1})$ becomes zero in $H^1(\Gal_{F_\frakp}, \psi_{\frakp,2}\psi_{\frakp,2}^{-1}\tensor B_\cris)$, which is exactly the definition of $e_{\frakp,22}$ being crystalline.
\end{proof}

In the next lemma, we show that for any $\frakp$ in $\Sigma\subset\Sigma_p$, $D^\frakp$ are all the same, and they parametrizes deformations in $D$ with constant $\frakp'$-Hodge--Tate--Sen weights for all $\frakp'\in\Sigma$.
For this, we briefly recall the definition of Hodge--Tate--Sen weights.
Let $\QQ_p^\cyc$ be the $p$-adic completion of $\QQ_p(\mu_{p^\infty})$, and $\Gamma$ be the Galois group $\Gal(\QQ_p^\cyc/\QQ_p)\isom \ZZ_p^\times$.
Sen's theory says that there is an equivalence of categories between the category of semi-linear $\CC_p$-representation of $\Gal_{\QQ_p}$ and the category of semi-linear $\QQ_p^\cyc$-representation of $\Gamma$.
 Let $D_\Sen$ denote Sen's functor, from the category of finite dimensional continuous $\QQ_p$-representations of $\Gal_{\QQ_p}$ to the category of semi-linear $\QQ_p^\cyc$-representations of $\Gamma$. 
Recall that a semi-linear $\QQ_p^\cyc$-representation comes equipped with a $\QQ_p^\cyc$-linear endomorphism $\phi$, called \emph{Sen endomorphism} (\cite[Theorem 4]{Sen81}).
Then the Hodge--Tate--Sen weights of a $\QQ_p$-representation $W$ of $\Gal_{\QQ_p}$ are by definition the eigenvalues of the Sen endomorphism on $D_{\Sen}(W)$.
Since $D_{\Sen}$ is an exact functor, we have
\[
 D_\Sen \colon \Ext^1_{\Gal_{\QQ_p}}(W,W) \rightarrow \Ext^1_\Gamma(D_\Sen(W), D_\Sen(W)).
\]

\begin{lem} \label{lem:constHT}
 Let $\widetilde{V} \in D(\bar{k}(x)[\varepsilon]/\varepsilon^2)$.
 If $\widetilde V$ lies in the subspace $D^\frakp(\bar{k}(x)[\varepsilon]/\varepsilon^2)\subset D(\bar{k}(x)[\varepsilon]/\varepsilon^2)$ for some $\frakp\in\Sigma$, then $\widetilde V$ has constant $\frakp'$-Hodge--Tate--Sen weights for all $\frakp'\in\Sigma$.
\end{lem}
\begin{proof}
Let $\frakp'\in\Sigma$.
 The self-extension space $\Ext^1_{\Gal_{F_{\frakp'}}}(V,V)$ decomposes into four terms and 
\[
 D_\Sen \colon 
\Directsum_{i,j=1}^2\Ext^1_{\Gal_{F_{\frakp'}}}(\psi_{\frakp',j},\psi_{\frakp',i}) 
\rightarrow  
\Directsum_{i,j=1}^2\Ext^1_\Gamma(D_\Sen(\psi_{\frakp',j}),D_\Sen(\psi_{\frakp',i}))
\]
 preserves the corresponding direct summands.
 Since $\psi_{\frakp',1}$ and $\psi_{\frakp',2}$ have distinct Hodge--Tate weights, $D_\Sen(\psi_{\frakp',i}e_{\frakp',ij})=0$ for $i\neq j$.

 By Lemma~\ref{lem:22crys}, $e_{\frakp',22}\psi_{\frakp',2}$ is a crystalline extension; in particular it is Hodge--Tate.
 Hence $D_\Sen(e_{\frakp',22}\psi_{\frakp',2})=0$.
 Namely, $\widetilde V$ has a constant $\frakp'$-Hodge--Tate weight $\frac{w+k_{\tau_{\frakp'}}-2}2$.
 
 On the other hand, by the definition of $D^\frakp$, $\widetilde V$ has a constant $\frakp$-Hodge--Tate weight $\frac{w-k_{\tau_{\frakp}}}2$, which is not equal to $\frac{w+k_{\tau_{\frakp}}-2}2$ by the assumption that $f$ has cohomological weight.
 Hence by condition~(\ref{cond:HT}) in the deformation functor $D$, which says that the sum of the two $\frakp'$-Hodge--Tate--Sen weights is independent of $\frakp'\in\Sigma_p$, we conclude that $\widetilde V$ has constant $\frakp'$-Hodge--Tate weight $\frac{w-k_{\tau_{\frakp'}}}2$ and $\frac{w+k_{\tau_{\frakp'}}-2}2$ for all $\frakp'\in\Sigma$.
\end{proof}

Because of the above Lemma, we write $D^\Sigma$ for the functor $D^\frakp$ for any $\frakp\in \Sigma$.

In the following, we would like to compare the dimension of $D^0(\bar{k}(x)[\varepsilon]/\varepsilon^2)$ and $D^\Sigma(\bar{k}(x)[\varepsilon]/\varepsilon^2)$.
Note that we have maps
\[
 \Ext^1_{\Gal_F}(V,V)
\xrightarrow{\Directsum\loc_\frakp}
\Directsum_{\frakp\in\Sigma_p}\Ext^1_{\Gal_{F_\frakp}}(V,V) 
\xrightarrow{\Directsum D_{\Sen,\frakp}}
\Directsum_{\frakp\in\Sigma_p}\Ext^1_{\Gamma}(D_\Sen(\left.V\right|_{\Gal_{F_\frakp}}),D_\Sen(\left.V\right|_{\Gal_{F_\frakp}}) ).
\]
Since we assumed $f$ has cohomological weight, for any $\frakp\in\Sigma_p$, $D_\Sen(\left.V\right|_{\Gal_{F_\frakp}})$ is the direct sum of its $1$-dimensional $\left(\phi=\frac{w-k_{\tau_\frakp}}2\right)$-eigenspace and $\left(\phi=\frac{w+k_{\tau_\frakp}-2}2\right)$-eigenspace.
 
\begin{prop} \label{prop:D0=ker}
 $D^0(\bar{k}(x)[\varepsilon]/\varepsilon^2)$ is the kernel of 
\[
 D^\Sigma(\bar{k}(x)[\varepsilon]/\varepsilon^2)
 \rightarrow 
\Directsum_{\frakp\notin\Sigma} \Ext^1_\Gamma(D_\Sen(\left.V\right|_{\Gal_{F_\frakp}})^{\phi=\frac{w-k_{\tau_\frakp}}2}, D_\Sen(\left.V\right|_{\Gal_{F_\frakp}})^{\phi=\frac{w-k_{\tau_\frakp}}2}).
\]
\end{prop}
\begin{proof}
 By Lemma~\ref{lem:constHT}, $\widetilde{V}$ already has constant $\frakp$-Hodge--Tate--Sen weights for all $\frakp\in\Sigma$. 
 The kernel of this map consists of $\widetilde{V}\in D^\Sigma(\bar{k}(x)[\varepsilon]/\varepsilon^2)$ which has a constant $\frakp$-Hodge--Tate--Sen weight $\frac{w-k_{\tau_\frakp}}2$ for $\frakp\notin\Sigma$.
 Then by condition~(\ref{cond:HT}) in the deformation functor $D$, $\widetilde V$ also has a constant $\frakp$-Hodge--Tate weight $\frac{w+k_{\tau_{\frakp}}-2}2$ for $\frakp\notin \Sigma$.
Namely $\widetilde V$ has constant $\frakp$-Hodge--Tate--Sen weights for all $\frakp\in\Sigma_p$.
\end{proof}

\begin{cor} \label{cor:codim}
 $\dim_{\bar{k}(x)} D^\Sigma(\bar{k}(x)[\varepsilon]/\varepsilon^2) - \dim_{\bar{k}(x)} D^0(\bar{k}(x)[\varepsilon]/\varepsilon^2) \leq d-\#\Sigma$.
\end{cor}
\begin{proof}
 Sen's theory says that $\Ext^1_\Gamma(D_\Sen(\left.V\right|_{\Gal_{F_\frakp}})^{\phi=\frac{w-k_{\tau_\frakp}}2}, D_\Sen(\left.V\right|_{\Gal_{F_\frakp}})^{\phi=\frac{w-k_{\tau_\frakp}}2})$ is isomorphic to 
\[
 \Ext^1_{\Gal_{\QQ_p}}(\CC_p, \CC_p)\isom H^1(\Gal_{\QQ_p},\CC_p), 
\]
which has dimension $1$.
Then the upper bound on the codimension of $D^0(\bar{k}(x)[\varepsilon]/\varepsilon^2)$ in $D^\Sigma(\bar{k}(x)[\varepsilon]/\varepsilon^2)$ follows from Proposition~\ref{prop:D0=ker}.
\end{proof}

Now we can deduce Theorem~\ref{thm:4to1}.
\begin{proof}[Proof of Theorem~\ref{thm:4to1}]
Corollary~\ref{cor:codim} and Lemma~\ref{lem:R^p} together says that\[								
\dim D(\bar{k}(x)[\varepsilon]/\varepsilon^2) - \dim D^0(\bar{k}(x)[\varepsilon]/\varepsilon^2) 
\leq d+1-\#\Sigma.			       
\] 
By Proposition~\ref{prop:R=T} (2), this implies that 
\[
 \dim T_x\cE - \dim T_x\cE_{\wt(x)} \leq d+1-\#\Sigma.
\] 
Since $\cE$ is of dimension $d+1$, $T_x\cE$ has dimension at least $d+1$, we conclude that $\dim T_x\cE_{\wt(x)}\geq\#\Sigma$.
\end{proof}


\bibliographystyle{amsalpha}
\bibliography{ram_of_Hilbert_eigenvariety}

\end{document}